\documentclass[reqno,a4paper]{amsart}

\usepackage[latin1,utf8]{inputenc}
\usepackage[english]{babel}
\usepackage{eucal,amsfonts,amssymb,amsmath,amsthm,epsfig,mathrsfs}
\usepackage{cancel,soul}
\usepackage{color}
\usepackage{amscd,amsxtra}
\usepackage{enumerate}
\usepackage{latexsym}
\usepackage{bm}

\usepackage{tikz}

\allowdisplaybreaks

\newcounter{ipotesi}

\Alph{ipotesi}
 \makeatletter \@addtoreset{equation}{section}

\makeatother \makeatletter
\newtheorem{thm}{Theorem}[section]
\newtheorem{hyp}[thm]{Hypotheses}{\rm}
{\rm}
\newtheorem{lemm}[thm]{Lemma}
\newtheorem{cor}[thm]{Corollary}

\newtheorem{prop}[thm]{Proposition}
\newtheorem{defi}[thm]{Definition}
\newtheorem{rmk}[thm]{Remark}{\rm}

\newcounter{parentenv}

\newcommand{\R}{{\mathbb R}}

\newcommand{\E}{{\mathbb E}}
\newcommand{\N}{{\mathbb N}}

\newcommand{\K}{{\mathcal{K}}}

\newcommand{\D}{{\nabla}}
\newcommand{\J}{{\mathcal{D}}}

\newcommand{\eps}{\varepsilon}
\newcommand{\ra}{\rightarrow}

\renewcommand{\tilde}[1]{\widetilde{#1}}

\newcommand{\Dom}{{\operatorname{Dom}}}

\newcommand{\Id}{{\operatorname{Id}}}

\newcommand{\set}[1]{{\left\{#1\right\}}}
\newcommand{\pa}[1]{{\left(#1\right)}}
\newcommand{\sq}[1]{{\left[#1\right]}}
\newcommand{\gen}[1]{{\left\langle #1\right\rangle}}
\newcommand{\abs}[1]{{\left|#1\right|}}
\newcommand{\norm}[1]{{\left\|#1\right\|}}
\newcommand{\scal}[2]{{\left\langle #1,#2\right\rangle}}

\newcommand{\eqsys}[1]{{\left\{\begin{array}{ll}#1\end{array}\right.}}
\newcommand{\tc}{\, \middle |\,}

\newcommand{\qc}{\mathbb{P}\mbox{-a.s.}}

\makeatletter
\@namedef{subjclassname@2020}{%
  \textup{2020} Mathematics Subject Classification}
\makeatother

\begin{document}

\frenchspacing

\title[$L^p$-$L^q$ estimates for transition semigroups]{$L^p$-$L^q$ estimates for transition semigroups associated to dissipative stochastic systems}

\author[L. Angiuli, D. A. Bignamini and S. Ferrari]{Luciana Angiuli, Davide A. Bignamini, Simone Ferrari$^*$}\thanks{$^*$Corresponding author}

\address{L.A. \& S.F.: Dipartimento di Matematica e Fisica ``Ennio De Giorgi'', Universit\`a del Salento, Via per Arnesano snc, 73100 LECCE, Italy}
\email{\textcolor[rgb]{0.00,0.00,0.84}{luciana.angiuli@unisalento.it}}
\email{\textcolor[rgb]{0.00,0.00,0.84}{simone.ferrari@unisalento.it}}

\address{D.A.B.: Dipartimento di Scienza e Alta Tecnologia (DISAT), Universit\`a degli Studi dell'In\-su\-bria, Via Valleggio 11, 22100 COMO, Italy}
\email{\textcolor[rgb]{0.00,0.00,0.84}{da.bignamini@uninsubria.it}}


\date{\today}

\begin{abstract}
In a separable Hilbert space, we study  supercontractivity and ultracontractivity properties for a transition semigroups associated with a stochastic partial differential equations. This is done in terms of exponential integrability of Lipschitz functions and some logarithmic Sobolev-type inequalities with respect to invariant measures.
The abstract characterization results concerning the improving of summability can be applied to transition semigroups associated to a stochastic reaction-diffusion equations.
\\

\noindent {\bf Acknowledgments.} The authors are members of GNAMPA (Gruppo Nazionale per l'Analisi Matematica, la Probabilit\`a le loro Applicazioni) of the Italian Istituto Nazionale di Alta Matematica (INdAM).
\end{abstract}

\maketitle

\section{Introduction}
Let $H$ be a separable Hilbert space with inner product $\langle\cdot,\cdot\rangle_H$ and induced norm $\norm{\cdot}_H$. We are interested in studying the stochastic partial differential equation
\begin{align}\label{SPDE}
\eqsys{dX(t)=[AX(t)+F(X(t))]dt+RdW(t) \qquad\;\, & t>0,\\
X(0)=x \in H,}
\end{align}
where $\{W(t)\}_{t\geq 0}$ is a $H$-cylindrical Wiener process defined on a normal filtered probability space $(\Omega,\mathcal{F},\{\mathcal{F}_t\}_{t \ge 0},\mathbb{P})$ (see e.g. \cite[Section 4.1.2]{DA-ZA3} for details on cylindrical Wiener processes in infinite dimension), $F:\Dom(F)\subseteq H\ra H$ is a regular enough function, $A:\Dom(A)\subseteq H\ra H$ is a linear closed operator and $R:H\ra H$ is a linear bounded, self-adjoint and positive operator. Under suitable assumptions, problem \eqref{SPDE} has a unique generalized mild solution $\{X(t,x)\}_{t\geq 0}$ that allows us to consider the transition semigroup $\{P(t)\}_{t\geq 0}$ associated to \eqref{SPDE} defined by
\[
(P(t)\varphi)(x):=\E\left[\varphi(X(t,x))\right]=\int_\Omega\varphi(X(t,x)(\omega))\mathbb{P}(d\omega),
\]
for every $t>0$, $x \in H$ and any Borel bounded function $\varphi:H\rightarrow \R$. Moreover, assuming some additional hypotheses on $A,F$ and $R$, it is possible to prove the existence and uniqueness of a probability invariant measure $\nu$ for $\{P(t)\}_{t\geq 0}$, namely a Borel probability measure on $H$ such that the equality
\begin{equation}\label{invariance_intro}
\int_HP(t)\varphi d\nu=\int_H\varphi d\nu, \qquad\;\, t>0,
\end{equation}
holds true for every bounded and continuous function $\varphi:H\rightarrow \R$. By the invariance of $\nu$, the semigroup $\{P(t)\}_{t\geq 0}$ is uniquely extendable to strongly continuous semigroups $\{P_p(t)\}_{t\ge 0}$ in $L^p(H,\nu)$ for every $p\geq 1$. Such semigroups are consistent, i.e. for every $f\in L^p(H, \nu)\cap L^q(H, \nu)$ it holds that $P_p(t)f=P_q(t)f$. For this reason, we will omit the subscript $p$, if no confusion can arise, and we continue to denote them as $\{P(t)\}_{t\geq 0}$.


The best known, and most studied, of these properties is the case $1<p<q<+\infty$ and $\|P(t)\|_{\mathcal{L}(L^p(H, \nu),L^q(H, \nu))}\le 1$ for any $t \ge \bar{t}= \bar{t}(p,q)>0$. In this case $\{P(t)\}_{t\ge 0}$ is said to be hypercontractive and lots of results concerning this property are well known both in finite and infinite dimension (see \cite{ane,AngLorLun,GRO1,Gross75,Gross93,ledoux,wang}).  Classical results in this direction concern with the connection between the hypercontractivity property, logarithmic Sobolev inequality (see \cite{GRO1,Gross75,Gross93}), and estimates for Gaussian tails of Lipschitz functions (see \cite[Chapter 5]{Bak-Gen-Led2014} and \cite[Sections 2.2 and 2.3]{ledoux2}).

Here, we are interested in investigating connections of this type in infinite dimension in the case of stronger improving summability properties and in characterizing it in terms of functional inequalities and integrability properties of exponential functions with respect to the invariant measure $\nu$.


Indeed, in a ``scale'' of summability improving properties, the next in line, stronger than the hypercontractivity, is the supercontractivity which means that for any $t>0$, the semigroup $P(t)$ belongs to $\mathcal{L}(L^p(H, \nu),L^q(H, \nu))$ for any $1<p<q<\infty$ and $\|P(t)\|_{\mathcal{L}(L^p(H, \nu),L^q(H, \nu))}\le 1$ for any $p,q$ and $t$ as before.
In finite dimension, there exist a quite large literature on this topic, see for instance \cite{AddAngLor,AngLor,lorbook,Roc_Wan}. In \cite{AngLor} the authors prove that if $H=\R^n$, the supercontrativity property of $\{P(t)\}_{t\geq 0}$ is equivalent to a stronger version (or defective, see \cite[Section 2.4]{ledoux2}) of the logarithmic Sobolev inequality and to exponential integrability of Lipschitz functions.  In infinite dimension a relation between defective logarithmic Sobolev inequality and the exponential integrability of Lipschitz functions is known, see \cite[Section 2.4]{ledoux2}.
However, to the best of our knowledge, the characterization result in \cite{AngLor} is confined to the finite dimensional case.
The goal of this paper is to improve some results in \cite[Chapter 5]{Bak-Gen-Led2014} and \cite[Chapter 2]{ledoux2} in our setting to obtain an infinite dimensional version of the characterization results as in \cite{AngLor}. In a theoretical framework, we are also able to characterize the ultraboundedness of $\{P(t)\}_{t \ge 0}$, i.e. $\|P(t)\|_{\mathcal{L}(L^p(H,\nu)), L^\infty(H,\nu))}<+\infty$ for any $t>0$ and $p>1$.

We underline that the abstract results of this paper are applicable to the transition semigroup $P(t)$ associated to the generalized mild solution of the well-studied stochastic reaction-diffusion equation
\begin{align}\label{cerrai}
\begin{cases}
dX(t)= \left[\frac{\partial^2}{\partial\xi^2}X(t) - b(X(t))\right]dt + dW(t), & t > 0,\\
X(t)(0)= X(t)(1), & t > 0,\\
X(0) = x\in L^2([0,1],\lambda)
\end{cases}
\end{align}
where $\lambda$ is the Lebesgue measure on $[0,1]$ and $b:\R\ra\R$ is a smooth enough polynomial function, see Section \ref{RDE} for details and \cite{CER1, DA-DE-GO1} for further results on this problem.  In particular, we point out that the characterization results complement the hypercontractivity result obtained in \cite{DA-DE-GO1} and the ultraboundness result obtained in \cite{DaPROWA09}. In particular in \cite{DaPROWA09} it is showed that $P(t)(L^2(H,\nu))\subseteq L^\infty(H,\nu)$. Here in Example \ref{RDE}, thanks to Proposition \ref{ultra} we show that
\[
P(t)(L^p(H,\nu))\subseteq L^\infty(H,\nu),
\]
for every $p>1$.


 We underline that it is natural to wonder whether it is possible to take advantage of the results in finite dimension and to extend them to the infinite dimension, letting the dimension to infinity. It must be emphasised that this technique does not work directly in our situation and the results obtained in this paper are not trivial extension to those proved in finite dimension. Indeed, even if the semigroup $\{P(t)\}_{t \ge 0}$ can be approximated in some sense by a sequence of finite dimensional semigroups $\{P_n(t)\}$, $n\in \N$, we cannot relate the improving summability properties of $\{P(t)\}$ in terms of those of $\{P_{n}(t)\}$ and viceversa. This is due essentially to the fact that in many cases, the relation between the invariant measure $\nu$  associated to $\{P(t)\}_{t \ge 0}$ and the sequence of the invariant measures $\nu_n$ associated to $\{P_n(t)\}$ is not known. This is, for instance the reason why it is not possible to prove the analogous of \cite[Theorem 3.1]{AngLor} to our setting. We conclude observing that another obstacle to overcome in our setting relies on the fact that a classical tool that fails in the transition from the finite dimension to the infinite one, and that in the first case allows to prove $L^p$-$L^q$ estimates, is the use of Lyapunov functions and some related estimates that depend explicitly on the dimension and that cannot be passed to the limit (see Remark \ref{BohBoh}).

The paper is organized in the following way. In Section \ref{preli} we collect our standing assumptions and recall some basic results. In Sections \ref{close} we prove that the operator
\[
R\nabla:C^1_b(H)\subseteq L^2(H,\nu)\ra L^2(H,\nu;H),
\]
where $\nabla$ is the Fr\'echet gradient and $C^1_b(H)$ is the set of bounded and Fr\'echet differentiable function with bounded and continuous derivative, is closable.
This property allows us to define the Sobolev space $W^{1,2}_R(H,\nu)$ as the domain of its closure. We point out that such kind of results has been already proved using more restrictive conditions in \cite{BF22,DA-DE-GO1}.

In Section \ref{expo}, we prove a logarithmic Sobolev inequality associated to $\nu$ that thanks to \cite[Theorem 3.7]{Gross93} yields the hypercontractivity of $\{P(t)\}_{t\geq 0}$.
Furthermore, we prove a less standard Fernique-type result for $\nu$, both with respect to the norm of $H$ and with respect to the norm $\|R^{-1}\cdot\|_H$ (see Theorems \ref{code} and \ref{codeH}).
In Section \ref{SUPER}, we study the supercontractivity of ${P(t)}_{t\ge 0}$ and its connections with stronger logarithmic Sobolev-type inequalities and exponential integrability of Lipschitz functions with respect to $\nu$. 
We stress that the results contained in Theorems \ref{code1} and \ref{code1H} are stronger than those in \cite[Section 2.4]{ledoux2}. This is due to the explicit estimates for the coefficients in \eqref{epslog} that are provided in this paper. Further in Theorem \ref{sup_thm}, assuming $\nu(R(H))=1$, we characterize the supercontractivity of $\{P(t)\}_{t\ge 0}$. 

In Section \ref{IDP0} we prove a characterization result for the ultraboundeness of $\{P(t)\}_{t\ge 0}$ (see Theorem \ref{ultra}) and we provide sufficient conditions on $F$ (Hypotheses \ref{super-dissiR}) that guarantee such property (see Proposition \ref{C-ultra}).
Finally, Section \ref{exp0} is devoted to provide explicit examples of stochastic partial differential equations whose coefficients satisfy our assumptions and to which our results can be applied.

\section*{Notations}
Let $\K_1$ and $\K_2$ be two Banach spaces equipped with the norms $\norm{\cdot}_{\K_1}$ and $\norm{\cdot}_{\K_2}$, respectively. Let $H$ be a Hilbert space equipped with inner product $\scal{\cdot}{\cdot}_{H}$ and associated norm $\norm{\cdot}_H$.

We denote by $\mathcal{L}(\K_1;\K_2)$ the space of linear and continuous maps from $\K_1$ to $\K_2$, if $\K_1=\K_2$ we write $\mathcal{L}(\K_1)$.

We denote by $\mathcal{B}(\K_1)$  the family of the Borel subsets of $\K_1$. $B_b(\K_1;\K_2)$ is the set of the bounded and Borel measurable functions from $\K_1$ to $\K_2$. If $\K_2=\R$ we simply write $B_b(\K_1)$. $C_b(\K_1;\K_2)$ (${\rm BUC}(\K_1;\K_2)$, respectively) is the space of bounded and continuous (uniformly continuous, respectively) functions from $\K_1$ to $\K_2$. If $\K_2=\R$ we write $C_b(\K_1)$ (${\rm BUC}(\K_1)$, respectively). $C_b(\K_1;\K_2)$ and ${\rm BUC}(\K_1;\K_2)$ are Banach spaces if endowed with the norm
\[
\norm{f}_{\infty}=\sup_{x\in \K_1}\|f(x)\|_{\K_2}.
\]


Let $X$ be a Banach space and let $A:{\rm Dom}(A)\subseteq X\ra X$ be a linear operator. If $Y$ is a Banach space continuously embedded in $X$, we call part of $A$ in $Y$ the operator $A_{Y}:\Dom(A_Y)\subseteq Y\ra Y$ defined as
\begin{align*}
\Dom(A_Y)&:=\{y\in \Dom(A)\cap Y\,|\, Ay\in Y\};\\
A_{Y}y&:=Ay,\qquad y\in \Dom(A_Y).
\end{align*}
If $A:{\rm Dom}(A)\subseteq H\ra H$ is a linear operator, we denote with $\mathcal{F}C^k_b(H)$ the space of real-valued, continuous, cylindrical along $\Dom(A^*)$ and bounded functions on $H$ whose derivatives up to order $k$ are continuous and bounded. More precisely,
\begin{align*}
\mathcal{F}C^k_b(H):=\set{f:H\ra\R\tc\begin{array}{c}
\text{there exist }n\in\N,\ \varphi\in C_b^k(\R^n)\text{ and }a_1,\ldots,a_n\in \Dom(A^*)\\
\text{such that }f(x)=\varphi(\gen{x,a_1}_H,\ldots,\gen{x,a_n}_H)\text{ for any }x\in H
\end{array}}.
\end{align*}

Let $(\Omega,\mathcal{F},\mathbb{P})$ be a probability space and let $H$ be a separable Hilbert space. Let $\xi:(\Omega,\mathcal{F},\mathbb{P})\ra (H,\mathcal{B}(H))$ be a random variable, we denote 
by
\[
\mathbb{E}[\xi]:=\int_\Omega \xi(\omega)  \mathbb{P}(d\omega)
\]
the expectation of $\xi$ with respect to $\mathbb{P}$. 
Let $\{Y(t)\}_{t\geq 0}$ be a $\K$-valued stochastic process we say that $\{Y(t)\}_{t\geq 0}$ is continuous if the map $Y(\cdot):[0,+\infty)\ra H$ is continuous $\mathbb{P}$-a.e.

For any $T>0$ and $p\geq 1$ the set
$C_p([0,T];H)$ ($C_p((0,T];H)$, respectively) denotes the space of progressive measurable $H$-valued processes $\{Y(t)\}_{t\in [0,T]}$ belonging to $L^p(\Omega;C([0,T];H))$ ($L^p(\Omega;C((0,T];H)\cap L^\infty(0,T;H))$, respectively), endowed with the norm
\[
\norm{\{Y(t)\}_{t\in [0,T]}}_{C_p([0,T];H)}=\norm{\{Y(t)\}_{t\in [0,T]}}_{C_p((0,T];H)}:=\left(\E\left[\sup_{t\in [0,T]}\norm{Y(t)}_H^p\right]\right)^{1/p}
\]
see \cite[Section 6.2]{CER1} for further details.

\section{Basic hypotheses and preliminary results}\label{preli}

In the whole paper, $H$ will be a separable Hilbert space with inner product $\langle\cdot,\cdot\rangle_H$ and induced norm $\norm{\cdot}_H$. Here we state the basic assumptions on the coefficients of the stochastic partial differential equation
\begin{align*}
\eqsys{dX(t)=[AX(t)+F(X(t))]dt+RdW(t) \qquad\;\, & t>0,\\
X(0)=x \in H,}
\end{align*}
and we recall some essential results about its solvability and the definition of its associated transition semigroup.

\begin{hyp}\label{main}
\begin{enumerate}[\rm(i)]
\item There exists a separable Banach space $E$, with norm $\norm{\cdot}_E$, continuously and densely embedded in $H$, such that $E\subseteq \Dom(F)$ and $F(E)\subseteq E$ where $F:\Dom(F)\subseteq H\ra H$;

\item $R \in \mathcal{L}(H)$ is a linear bounded, self-adjoint and positive operator;

\item $A:\Dom(A)\subseteq H\ra H$ is a self-adjoint operator that generates a strongly continuous semigroup $e^{tA}$ on $H$ and $A_E$ (the part of $A$ in $E$) generates an analytic semigroup $e^{tA_E}$ on $E$. There exists $\zeta_A\in\R$ such that $A-\zeta_A\Id_{H}$ is dissipative in $H$ and $A_E-\zeta_A\Id_{E}$ is dissipative in $E$;

\item the stochastic convolution process $\{W_A(t)\}_{t\ge 0}$, defined as
\begin{equation}\label{stoconv}
W_{A}(t):=\int^t_0e^{(t-s)A} R dW(s),\qquad t\geq 0,
\end{equation}
belongs to $C_p([0,T],E)$ for any $T>0$ and $p\geq 1$  and
\begin{equation}\label{num}
\sup_{t\geq 0}\E\left[\norm{W_A(t)}^p_E\right]<+\infty;
\end{equation}

\item there exists $\zeta_F\in\R$ such that $F-\zeta_F\Id_{H}$ is $m$-dissipative in $H$ and $F_{|_E}-\zeta_F\Id_{E}$ is dissipative in $E$. Moreover, $F_{|_E}:E\ra E$ is Fréchet differentiable and there exist $m\in\N$ and $C>0$ such that for any $x\in E$

\begin{align}
\|F_{|_E}(x)\|_E &\leq C(1+\norm{x}^{2m+1}_E);\notag\\
\|\J F_{|_E}(x)\|_{\mathcal{L}(E;E)} &\leq C(1+\norm{x}^{2m}_E);\notag\\
\|\J F_{|_E}(x)h\|_{H} &\leq C(1+\norm{x}^{2m}_E)\norm{h}_H\qquad h\in E\label{dissiX};
\end{align}

\item $\zeta:=\zeta_A+\zeta_F<0$.
\end{enumerate}
\end{hyp}
\noindent We want to underline that Hypothesis \ref{main}(vi) is not a standard requirement asked in order  to get existence results of a solution to \eqref{SPDE} but it will be used to deduce the asymptotic behaviour of the semigroup that we are going to define.

The assumptions on $R$ allow to prove that the space $H_R:=R(H)$ is a separable Hilbert space with respect to the norm $\norm{x}_R:=\|R^{-1}x\|_H$ induced by the inner product
\begin{align}\label{inn-prod-hr}
\scal{x}{y}_R:=\langle R^{-1}x,R^{-1}y\rangle_H,\qquad x,y\in H_R.
\end{align}
Furthermore $H_R$ is a Borel measurable space, continuously embedded in $H$ and, as it is easy to prove, it holds that
\begin{gather}\label{tonino}
\|h\|\leq \|R\|_{\mathcal{L}(H)}\|h\|_{H_R},
\end{gather}
for any $ h \in H_R$, (see \cite{ABF21,BF22,BF20,BF23,BFFZ23}, \cite[Theorem 15.1]{KE1} and \cite[Appendix C]{LI-RO1} for further details). We need to define a family of function that we will use throughout the paper. We say that a function $g:H\to \R$ is $H_R$-Lipschitz if
\begin{align}\label{Joe}
|g(x+h)-g(x)|\leq L \|h\|_{R},\qquad
\end{align}
for any $x\in H$, $h\in H_R$ and some $L>0$. The best constant in \eqref{Joe} is denoted ${\rm{Lip}}_R(g)$. We denote by ${\rm{Lip}}_{H_R}(H)$ the set of $H_R$-Lipschitz functions from $H$ to $\R$.

Recall that, if $\Dom(F)=H$ and $x \in H$, a mild solution of \eqref{SPDE} is a $H$-valued stochastic process $\set{X(t,x)}_{t\geq 0}$ satisfying
\begin{equation}\label{mild}
X(t,x)=e^{tA}x+\int_0^te^{(t-s)A}F(X(s,x))ds+W_A(t),\qquad t\geq 0,\ \mathbb{P}\text{-a.s.},
\end{equation}
where  $\{W_A(t)\}_{t\geq 0}$ is defined in \eqref{stoconv}.

On the other hand, we say that $\{\tilde{X}(t,x)\}_{t\geq 0}$ is a generalized mild solution of \eqref{SPDE} if
\begin{enumerate}[\rm(i)]
\item for any $x\in E$ there exists a $E$-valued stochastic process $\{X(t,x)\}_{t\geq 0}$ satisfying \eqref{mild};
\item for any sequence $\{x_n\}_{n\in\N}\subseteq E$ converging to $x$ in $H$ and for any $T>0$ it holds
\begin{equation}\label{gmild}
\lim_{n\ra+\infty}\E\left[\sup_{t\in [0,T]}\|X(t,x_n)-\tilde{X}(t,x)\|_H\right]=0,
\end{equation}
where $\{X(t,x_n)\}_{t\geq 0}$ is a mild solution of \eqref{SPDE} with initial datum $x_n$.
\end{enumerate}
We point out that, if $x\in E$ the generalized mild solution of \eqref{SPDE} actually coincide with the mild solution. In view of this fact, when there is no confusion, we will use the same notation for both of them.

In the following theorem we collect some known results concerning existence, uniqueness and regularity of a mild and generalized mild solution to \eqref{SPDE} whose proof can be found in \cite[Proposition 2.5]{ABM} and  \cite[Theorem 3.9, Corollary 3.10 and Proposition 3.13]{BI1}.

\begin{thm}\label{Genmild}
Assume Hypotheses \ref{main} hold true and let $T>0$ and $p \ge 1$. The following statements hold true.
For any $x\in E$, \eqref{SPDE} has a unique mild solution $\{X(t,x)\}_{t\geq 0}$ belonging to $C_p([0,T];H)\cap C_p((0,T];E)$. Moreover it is Gateaux differentiable as a function from $E$ to $C_p([\eps,T];E)$ for any $0<\eps<T$ and, for any $x \in E$, its Gateaux derivative $\{\J_GX(t,x)h\}_{t\geq 0}$ is the unique mild solution of
\begin{align}\label{varprimaa}
\eqsys{dY_x(t)=[A+\J F(X(t,x))]Y_x(t)dt, & t>0,\\
Y_x(0)=h \in E}
\end{align}
i.e.
\begin{equation}\label{mild_der}
\J_GX(t,x)h=e^{tA}h+\int^t_0e^{(t-s)A}\J F(X(s,x))\J_GX(s,x)h ds,\qquad x,h\in E,\ t>0,
\end{equation}
and it satisfies
\begin{align}\label{stimaE}
\norm{\J_GX(t,x)h}_{E}\leq e^{\zeta t}\norm{h}_{E},\qquad x,h\in E,\ t>0,
\end{align}

and, for every $p>0$ there exists a real positive random variable $K_p$ such that $\mathbb{E}[K_p]<+\infty$ and
\begin{align}\label{momeX}
\norm{X(t,x)}_\mathcal{O}^p\leq  K_p\left(1+\norm{x}_\mathcal{O}^p\right),\qquad\qc
\end{align}
being either $\mathcal{O}=H$ or $\mathcal{O}=E$.
\end{thm}


\noindent We emphasize that, regardless of explicit mention, all the (in)equalities in the preceding theorem and subsequent content are implied to be valid for $\mathbb{P}$-a.e. $\omega\in\Omega$. Furthermore, when we refer to uniqueness, we are specifically indicating pathwise uniqueness.

Theorem \ref{Genmild} allows to define the transition semigroup $\{P(t)\}_{t\geq 0}$ ($P(t)$ in short) defined as
\begin{align}\label{trans}
(P(t)\varphi)(x):=\E[\varphi(X(t,x))],\qquad x\in H,\ t\geq 0,\ \varphi\in B_b(H)
\end{align}
and, more precisely, assumption \eqref{num} and Hypothesis \ref{main}(vi) guarantee the existence of a unique probability invariant measure $\nu$ for $P(t)$, i.e. a probability measure $\nu$ on $H$ such that
\begin{align*}
\int_HP(t)\varphi d\nu=\int_H\varphi d\nu
\end{align*}
for any $t>0$ and $\varphi\in C_b(H)$.

\begin{rmk}{\rm
Note that, as proved in \cite[Theorem 3.19]{BI1}, Hypotheses \ref{main} imply that
\begin{enumerate}[\rm(i)]
\item $\nu(E)=1$;
\item for any $p\geq 1$, it holds that
$$\int_{E} \|x\|_E^p \nu (dx)<+\infty;$$
\item for any $\varphi \in C_b(\mathcal O)$ and $x \in \mathcal{O}$, being $\mathcal{O}$ either $E$ or $H$,
\begin{equation}\label{trans1}
\lim_{t \to +\infty}(P(t)\varphi)(x)= \int_{\mathcal{O}}\varphi d\nu.
\end{equation}
\end{enumerate}}
\end{rmk}

\section{Closability of the gradient along $H_R$}\label{close}

The aim of this section is to introduce suitable Sobolev spaces where our computations are set. 
The results that we prove extend those in \cite{DA-DE-GO1} (where the authors consider the case when $F$ is dissipative and $R$ is the identity operator), as well as the results contained in \cite{BF22} (where the domain of $F$ is the whole space $H$). To this aim we need some a priori estimate on $\norm{\J_GX(t,x)h}_R$ uniform with respect to $x \in E$. This can be obtained under a further compatibility condition between $A$ and $R$ together with either the Lipschitzianity of $F$ or some stronger dissipativity assumption on $F_{|_E}$.

Here we assume the following additional hypotheses:

\begin{hyp}\label{hyp2}
Besides Hypotheses \ref{main}, assume that $E\cap H_R$ is dense in $H_R$ and for any $t>0$,
\(
e^{tA}(H)\subseteq R(H).
\)
Moreover there exists $\gamma\in [0,1)$, $w>0$ and $M>0$ such that
\begin{equation}\label{num_1}
\|e^{tA}x\|_R\leq Me^{-wt}t^{-\gamma}\norm{x}_H,\qquad t>0,\ x\in H.
\end{equation}
\end{hyp}
Note that Hypotheses \ref{hyp2} is satisfied, for instance, if $R=(-A)^{-\beta}$ for some $\beta>0$ and $A$ generates an analytic semigroup of negative type in $H$. In this case, estimate \eqref{num_1} is proved in \cite[Proposition 2.1.1]{LUN1}.

\begin{lemm}\label{est_pri}
Assume Hypotheses \ref{hyp2} hold true. There exists a real positive random variable $K$ depending on $m,\gamma,w$ and $\zeta$ such that $\mathbb{E}[K]<+\infty$ and
\begin{align}\label{persobolev}
\norm{\J_GX(t,x)h}_R\leq K(1+\norm{x}_E^{2m})e^{-\min\{w,|\zeta|\}t}\max\{t^{-\gamma}, t^{1-\gamma}\}\norm{h}_H,\quad\qc
\end{align}
for any $x\in E$, $h\in H_R$ and  $t>0$.
\end{lemm}

\begin{proof}
Let $h\in E\cap H_R$. Taking the $\norm{\cdot}_R$ in \eqref{mild_der} and using Hypotheses \ref{hyp2}, we infer that
\begin{align}\label{S1-25}
\|\J_GX(t,x)h\|_R\leq \frac{Me^{-wt}}{t^{\gamma}}\norm{h}_H+\int^t_0\frac{Me^{-w(t-s)}}{(t-s)^{\gamma}}\norm{\J F(X(s,x))\J_GX(s,x)h}_H\,\,ds
\end{align}
for any $x\in E$, $h\in H$ and  $t>0$. Thanks to \eqref{dissiX}, \eqref{stimaE} and \eqref{momeX} we obtain
\begin{align}
\int^t_0\frac{Me^{-w(t-s)}}{(t-s)^{\gamma}}\|\J F(X(s,x))&\J_GX(s,x)h\|_H ds \leq C_1(1+\norm{x}_E^{2m})\norm{h}_H \int^t_0\frac{e^{-w(t-s)}}{(t-s)^{\gamma}}e^{\zeta s}ds\notag\\
& \leq C_1(1+\norm{x}_E^{2m})C_2\max(1,t^{1-\gamma})e^{-\min\{w,|\zeta|\}t}\norm{h}_H,\label{S2}
\end{align}
for some $C_1$ real positive random variable such that $\mathbb{E}[C_1]<+\infty$ and a positive constant $C_2$. Thus, taking \eqref{S1-25} and \eqref{S2}  into account we get \eqref{persobolev} for every $h\in E\cap H_R$. Finally recalling that $E\cap H_R$ is dense in $H_R$ we obtain the statement.
\end{proof}

The following result is a straightforward consequence of the previous lemma, but it will be fundamental at various points in the paper.

\begin{cor}
Under the assumptions of Proposition \ref{est_pri}, if Hypotheses \ref{main}(v) are satisfied with $m=0$ then
\begin{align}\label{Stima_sporchina}
\norm{\J_GX(t,x)h}_R\leq K e^{-\min\{w,|\zeta|\}t}\max\{t^{-\gamma}, t^{1-\gamma}\}\norm{h}_H=:\theta(t)\norm{h}_H,\qquad\qc
\end{align}
for any $x\in E$, $h\in H_R$ and some positive constant $K$ that depends only on the Lipschitz constant of $F$.
\end{cor}

Now, we assume a stronger dissipativity condition on $F$ to get a more precise decay estimate on $\norm{\J_GX(t,x)h}_R$.
\begin{hyp}\label{dissiR}
Beside Hypotheses \ref{hyp2}, assume that, $\J F_{|_E}(x)(H_R\cap E)\subseteq H_R$ for any $x\in E$ and there exists $\zeta_R \in \mathbb{R}$ such that
\begin{equation}\label{3.4}
\scal{[A_R+\J F(x)]h}{h}_R\leq \zeta_R\norm{h}_R^2, \quad\;\, x\in E, h\in\Dom(A_R),\qquad\qc
\end{equation}
where $A_R$ denotes the part of $A$ in $H_R\cap E$.
\end{hyp}

Here, using Hypotheses \ref{dissiR} we prove an alternative estimate for $\norm{\J_GX(t,x)h}_R$.

\begin{prop}
If Hypotheses \ref{dissiR} hold true, then
\begin{align}\label{stimaR}
\norm{\J_GX(t,x)h}_R\leq e^{\zeta_R t}\norm{h}_R,
\end{align}
for any $x\in E$, $h\in H_R$ and  $t>0$.
\end{prop}

\begin{proof}
It is not restrictive to assume that  $\{\J_G X(t,x)h\}_{t\geq 0}$ is a strict solution of \eqref{varprimaa}. Indeed, otherwise we proceed as in \cite[Proposition 3.6]{BI1} or \cite[Proposition 6.2.2]{CER1} approximating  $\{\J_G X(t,x)h\}_{t\geq 0}$ by means of sequences of more regular processes using the Yosida approximation of $A_R$. Moreover, estimates \eqref{stimaE} and \eqref{persobolev} yield that $\J_G X(t,x)h\in E\cap H_R$ for any $t>0$, $x\in E$ and $h\in E\cap H_R$.
Now, let us fix $x, h$ as before, multiplying \eqref{varprimaa} by $\J_G X(t,x)h$ and using estimate \eqref{3.4}, we obtain
\begin{align*}
\frac{1}{2}\frac{d}{ds}\|\J_G X(s,x)h\|^2_R=\langle [A_R+\J F(X(s,x))]\J_G X(s,x)h,\J_G X(s,x)h\rangle_R\leq \zeta_R\|\J_G X(s,x)h\|^2_R
\end{align*}
whence, integrating from $0$ to $t$ with respect to $s$ we get  \eqref{stimaR} for every $h\in E\cap H_R$. The claim follows by the fact that $E\cap H_R$ is dense in $H_R$.
\end{proof}

Now we define the functional spaces which will play a crucial role in this paper. The following notion of differentiability first appeared in \cite{GRO1} (see also \cite{KUO1}).

\begin{defi}
Let $R \in \mathcal{L}(H)$ be a linear bounded, self-adjoint and positive operator. We say that a function $\Phi:H\ra \R$ is $H_R$-differentiable at $x\in H$ if there exists  $L_x\in\mathcal{L}(H_R;\R)$ such that
\begin{align*}
\lim_{\norm{h}_R\ra 0}\frac{|\Phi(x+h)-\Phi(x)-L_xh|}{\norm{h}_R}=0
\end{align*}
$($see \eqref{inn-prod-hr}$)$.
If $L_x$ exists, then it is unique and we set $\J_R\Phi(x):=L_x$. We say that $\Phi$ is $H_R$-differentiable if $\Phi$ is $H_R$-differentiable at every $x\in H$.
\end{defi}
Note that if $R={\rm Id}_{H}$, then the $H_R$-differentiability coincides with the standard Fr\'echet differentiability, in this case we will drop the subscript $R$.

\begin{rmk}{\rm
Let $f:H\ra\R$ be a $H_R$-differentiable function. Since $H_R$ is a Hilbert space, by the Riesz representation theorem, for every $x\in H$ there exists a unique $l_x\in H_R$ such that
\[
\J_R f(x)h=\langle l_x,h\rangle_R,\qquad h\in H_R.
\]
We call $l_x$ the $H_R$-gradient of $f$ at $x\in H$ and we denote it by $\nabla_R f(x)$. If $R={\rm Id}_H$ then $\nabla_R$ is the classical Fréchet gradient and we simply write $\nabla$.}
\end{rmk}
The notion of $H_R$-differentiability is a sort of Fr\'echet differentiability along the directions of $H_R$ and it has been already considered in various papers (see, for example, \cite{ABF21,BF20,BF22,CL19,CL21}). The proof of the following result follows the same arguments used in the proof of \cite[Proposition 17]{BF20} or in \cite[Theorem 2.21]{BFFZ23}.

\begin{prop}\label{dalpha}
Assume Hypothesis \ref{main}(ii) holds true. If $f:H\ra\R$ is a Fr\'echet differentiable function with continuous derivative operator, then $f$ is also $H_R$-differentiable with continuous $H_R$-derivative operator and, for every $x\in H$, it holds $\nabla_Rf(x)=R^2\nabla f(x)$.
\end{prop}

It is well known that $P(t)$ can be extended to a strongly continuous semigroup in $L^p(H, \nu)$ denoted by $P_p(t)$.
By \cite[Theorem 1.1]{BI1}, the infinitesimal generator $N_2:\Dom(N_2)\subseteq L^2(H,\nu)\ra L^2(H,\nu)$ of $P_2(t)$ is the closure in $L^2(H, \nu)$ of the operator $\mathcal{N}_0$ defined on smooth cylindrical functions
 $\varphi \in  \mathcal{F}C^2_b(H)$ as follows
\[[\mathcal{N}_0\varphi](x):= \frac{1}{2}{\rm Tr}[R^2\nabla^2\varphi(x)] + \langle x, A^*\nabla\varphi(x)\rangle_H  + \langle F(x), \nabla \varphi(x)\rangle_H, \qquad x \in E.\]

\begin{lemm}
Assume Hypotheses \ref{main} hold true. For any $\varphi,\psi\in \mathcal{F}C^2_b(H)$ it holds
\begin{align}
[\mathcal{N}_0(\varphi\psi)](x) &=\varphi(x) [\mathcal{N}_0\psi](x)+\psi(x) [\mathcal{N}_0\varphi](x) +\langle \nabla_R\varphi(x),\nabla_R\psi(x)\rangle_R,\qquad x\in E.\label{N_2quadro}
\end{align}
Furthermore, for any $\psi\in \Dom(N_2)$, it holds
\begin{align}\label{Hazel}
\int_H \psi N_2\psi d\nu=-\frac{1}{2}\int_H \|\nabla_R\psi\|_R^2d\nu.
\end{align}
\end{lemm}

\begin{proof}
The fact that $\varphi\psi$ belongs to $\mathcal{F}C^2_b(H)$ and formula \eqref{N_2quadro} are immediate. In order to prove \eqref{Hazel} we start by recalling that since $\nu$ is an invariant measure and formula \eqref{N_2quadro} holds true, for any $\varphi\in \mathcal{F}C^2_b(H)\subseteq\Dom(N_2)$, then
\begin{align}\label{ext}
0=\int_H N_2(\varphi^2) d\nu=\int_H\pa{2\varphi N_2\varphi +\|\nabla_R\varphi\|_R^2}d\nu.
\end{align}
Since $\mathcal{F}C^2_b(H)$ is a core for $N_2$, by \eqref{ext} and the Young inequality, it follows that the map
\[\nabla_R: \mathcal{F}C^2_b(H)\subseteq \Dom(N_2)\ra L^2(H,\nu;H_R),\qquad  \varphi\mapsto \nabla_R\varphi,\]
is continuous and, consequently, it can be extended to functions $\varphi \in \Dom(N_2)$ (endowed with the graph norm). 
Now \eqref{Hazel} follows by a standard density argument.
\end{proof}

The next result is a technical lemma about the behaviour of the gradient of the semigroup $P_2(t)$ which will be useful to prove the closability of the gradient operator in $L^2(H, \nu)$.

\begin{lemm}
Assume Hypotheses \ref{main} hold true. For any $\varphi\in \mathcal{F}C^1_b(H)$ it holds that
\begin{align}\label{3dinotte}
\int_H|P_2(t)\varphi|^2d\nu+\int_0^t\int_H\|\nabla_R P_2(s)\varphi\|_R^2d\nu ds=\int_H|\varphi|^2d\nu.
\end{align}
\end{lemm}

\begin{proof}
Let $\varphi\in \mathcal{F}C^1_b(H)$, then
\begin{align}\label{caldissimo}
\frac{d}{ds}P_2(s)\varphi=N_2P_2(s)\varphi,\qquad s>0.
\end{align}
Multiplying both sides of \eqref{caldissimo} by $P_2(s)\varphi$, integrating on $H$ with respect to $\nu$, and taking into account \eqref{Hazel}, we find
\begin{align}\label{Carnevale}
\int_H\frac{d}{ds}|P_2(s)\varphi|^2 d\nu=-\int_H\|\nabla_R P_2(s)\varphi\|_R^2d\nu.
\end{align}
The thesis follows integrating \eqref{Carnevale} with respect to $s$ from $0$ to $t$.
\end{proof}

The next proposition is the main result of this section and allows us to define the Sobolev space $W_R^{1,2}(H,\nu)$ if either estimate \eqref{Stima_sporchina} or estimate \eqref{stimaR} hold true. This is the case, for instance, when either Hypotheses \ref{hyp2} with $m=0$ or Hypotheses \ref{dissiR} hold true.

\begin{prop}
If either Hypotheses \ref{hyp2} hold true with $m=0$ or Hypothesis \ref{dissiR} is satisfied, then the operator $\nabla_R:\mathcal{F}C^1_b(H)\subseteq L^2(H,\nu)\ra L^2(H,\nu;H_R)$ is closable.
\end{prop}

\begin{proof}
We will prove the claim only in the case when Hypotheses \ref{hyp2} with $m=0$ hold true, as the case in which Hypotheses \ref{dissiR} hold true follows using similar arguments. To this aim, we consider $\{\varphi_n\}_{n\in\N}\subseteq \mathcal{F}C^1_b(H)$ a sequence such that
\begin{align}
L^2(H,\nu)&\text{-}\lim_{n\ra+\infty}\varphi_n =0;\label{Hamilton}\\
L^2(H,\nu;H_R)&\text{-}\lim_{n\ra+\infty}\nabla_R\varphi_n =\Psi.\notag
\end{align}
From here onwards, for the remainder of the proof, we will assume that $t\in(0,1)$. By \eqref{3dinotte}, the continuity of the operator $P_2(t)$ and \eqref{Hamilton}, we have
\begin{align}\label{Orpheus}
\lim_{n\ra+\infty}\int_0^t\int_H\|\nabla_R P_2(s)\varphi_n\|_R^2d\nu ds=\lim_{n\ra+\infty}\pa{\int_H|\varphi_n|^2d\nu-\int_H|P_2(t)\varphi_n|^2d\nu}=0.
\end{align}
We claim that
\begin{align}\label{Euridice}
\lim_{n\ra+\infty}\int_0^t\int_H\|\nabla_R P_2(s)\varphi_n\|_R^2d\nu ds=\int_0^t\int_H\norm{\E\sq{(\J_G X(s,x))^*\Psi(X(s,x))}}_R^2\nu(dx)ds
\end{align}
where $(\J_G X(s,x))^*$ denotes the adjoint of $\J_G X(s,x)$ in $\mathcal{L}(H_R)$.
Indeed by \cite[Corollary 3.11]{BF20}
\begin{align*}
\nabla_R P_2(t)\varphi_n(x)=\E\sq{(\J_G X(t,x))^*\nabla_R\varphi_n(X(t,x))}.
\end{align*}
By the fact that $(\J_G X(t,x))^*$ satisfies \eqref{Stima_sporchina}, the invariance of $\nu$ (see \eqref{invariance_intro}) and \eqref{tonino}, we can estimate
\begin{align}
\int_0^t\int_H &\norm{\E\sq{(\J_G X(s,x))^*\nabla_R\varphi_n(X(s,x))}-\E\sq{(\J_G X(t,x))^*\Psi(X(s,x))}}_R^2\nu(dx)ds\notag\\
&\leq  \|R\|_{\mathcal{L}(H)}\int_0^t\int_H \theta(s)\norm{\E\sq{\nabla_R\varphi_n(X(s,x))-\Psi(X(s,x))}}_R^2\nu(dx)ds\notag\\
&\leq  \|R\|_{\mathcal{L}(H)}\int_0^t\int_H \theta(s)P_2(s)\norm{{\nabla_R\varphi_n(x)-\Psi(x)}}_R^2\nu(dx)ds\notag\\
&\leq \frac{K\|R\|_{\mathcal{L}(H)}\Gamma(1-\gamma)}{(\min\{\omega,|\zeta|\})^{1-\gamma}}\int_H \norm{{\nabla_R\varphi_n-\Psi}}_R^2d\nu,\label{Dragon}
\end{align}
where $K,\gamma,\omega$ and $\zeta$ are introduced in Hypotheses \ref{hyp2}, $\theta$ is defined in formula \eqref{Stima_sporchina} and $\Gamma$ is the gamma function, namely $\Gamma(z) := \int_{0}^{+\infty} r^{z-1} e^{-r} dr$.
So taking the limit as $n$ approaches infinity in \eqref{Dragon} we obtain
\begin{align*}
0\leq &\lim_{n\ra+\infty}\int_0^t\int_H\norm{\E\sq{(\J_G X(s,x))^*\nabla_R\varphi_n(X(s,x))}-\E\sq{(\J_G X(s,x))^*\Psi(X(s,x))}}_R^2\nu(dx)ds\\
\leq & \frac{K\|R\|_{\mathcal{L}(H)}\Gamma(1-\gamma)}{(\min\{\omega,|\zeta|\})^{1-\gamma}}\lim_{n\ra+\infty}\int_H \norm{{\nabla_R\varphi_n-\Psi}}_R^2d\nu=0.
\end{align*}
This prove \eqref{Euridice}. Combining \eqref{Orpheus} and \eqref{Euridice} we get
\begin{align*}
\int_0^t\int_H\norm{\E\sq{(\J_G X(s,x))^*\Psi(X(s,x))}}_R^2\nu(dx)ds=0.
\end{align*}
So for a.e. $s\in(0,t)$ (with respect to the Lebesgue measure) it holds
\begin{align}\label{Persefone}
\int_H\norm{\E\sq{(\J_G X(s,x))^*\Psi(X(s,x))}}_R^2\nu(dx)=0.
\end{align}
Thus, let $A\subseteq (0,t)$ be the set of measure zero in which \eqref{Persefone} fails to hold. By the monotone convergence theorem we infer that
\begin{align*}
0=&\int_H\norm{\E\sq{(\J_G X(s,x))^*\Psi(X(s,x))}}_R^2\nu(dx)\\
=&\sum_{i=1}^{+\infty}\int_H\abs{\E\sq{\gen{(\J_G X(s,x))^*\Psi(X(s,x)),h_i}_R}}^2\nu(dx)\\
=&\sum_{i=1}^{+\infty}\int_H\abs{\E\sq{\gen{\Psi(X(s,x)),\J_G X(s,x)h_i}_R}}^2\nu(dx)
\end{align*}
for any $s\in (0,t)\setminus A$, being $\{h_i\,|\, i\in\N\}$ an orthonormal basis of $H_R$.
So, for such $s$ and any $i\in\N$
\[\int_H\abs{\E\sq{\gen{\Psi(X(s,x)),\J_G X(s,x)h_i}_R}}^2\nu(dx)=0.\]
Thus for any $s\in (0,t)\setminus A$ and $i \in \N$
\begin{align}\label{est}
0\leq &\norm{P_2(s)(\gen{\Psi(\cdot),h_i}_R)}_{L^2(H,\nu)}=\norm{\E\sq{\gen{\Psi(X(s,\cdot)),h_i}_R}}_{L^2(H,\nu)}\notag\\
=&\norm{\E\sq{\gen{\Psi(X(s,\cdot)),h_i}_R}}_{L^2(H,\nu)}-\norm{\E\sq{\gen{\Psi(X(s,\cdot)),\J_G X(s,\cdot)h_i}_R}}_{L^2(H,\nu)}\notag\\
\leq & \norm{\E\sq{\gen{\Psi(X(s,\cdot)),h_i}_R}-\E\sq{\gen{\Psi(X(s,\cdot)),\J_G X(s,\cdot)h_i}_R}}_{L^2(H,\nu)}\notag\\
=&\norm{\E\sq{\gen{\Psi(X(s,\cdot)),h_i-\J_G X(s,\cdot)h_i}_R}}_{L^2(H,\nu)}.
\end{align}
The continuity of $s\mapsto\J_G X(s,\cdot)$ and the dominated convergence theorem allow us to conclude that the right hand side in \eqref{est} vanishes as $s$ approaches zero from the right. Consequently, taking the limit as $s \to 0^+$ in \eqref{est} we can conclude that
\[\|\gen{\Psi(\cdot),h_i}_R\|_{L^2(H,\nu)}=0, \qquad i \in \N.\]
Standard argument allows to conclude that $\Psi(x)=0$ for $\nu$-a.e $x\in H$ and this proves the closability of $\nabla_R:\mathcal{F}C^1_b(H)\subseteq L^2(H,\nu)\ra L^2(H,\nu;H_R)$.
\end{proof}

The previous result allows us to define the Sobolev space $W_R^{1,2}(H,\nu)$ as the domain of the closure of the operators $\nabla_R:\mathcal{F}C^1_b(H)\subseteq L^2(H,\nu)\ra L^2(H,\nu;H_R)$.

\section{Hypercontractivity and exponential integrability}\label{expo}

As already recalled, the semigroup $P(t)$ can be extended to a bounded and strongly continuous semigroup in $L^p(H, \nu)$ for any $p \ge 1$. Due to the consistence of such operators in the $L^p$-scale, we omit the dependence on $p$ and we continue to denote it by $P(t)$. In this section we are interested in proving the exponential integrability with respect to the invariant measure $\nu$ of $H_R$-Lipschitz functions. This can be obtained by means of some logarithmic Sobolev inequalities which are connected to the hypercontractivity of the semigroup $P(t)$ in $L^p(H, \nu)$ spaces.
 Using formula \eqref{trans}, the chain rule in \cite[Corollary 21]{BF20} and estimates \eqref{Stima_sporchina} and \eqref{stimaR} it can be proved that the semigroup satisfies the following gradient estimate:
\begin{align}\label{Azala}
\|\nabla_R P(t)\varphi\|_R^2\leq \psi(t)P(t)\|\nabla_R\varphi\|_R^2, \qquad \varphi\in \mathcal{F}C_b^1(H),\ t>0
\end{align}
where $\psi(t)=\theta(t)$ (see \eqref{Stima_sporchina}) if Hypotheses \ref{hyp2} are satisfied with $m=0$ whereas $\psi(t)= e^{2\zeta_R t}$ if Hypotheses \ref{dissiR}
are satisfied.
In this section, in addition to Hypotheses \ref{main} we assume either Hypotheses \ref{hyp2} with $m=0$ or Hypotheses \ref{dissiR} with $\zeta_R<0$.

In these cases, thanks to estimate \eqref{Azala}, the asymptotic behaviour of $P(t)$ as $t\to +\infty$, see \eqref{trans1}, and the fact that $\psi \in L^1((0,\infty))$, applying the classical idea of Deuschel and Stroock (see \cite{DS90}) we can prove a logarithmic Sobolev inequality with respect to the invariant measure $\nu$.

\begin{thm}\label{logsob_pro}
Assume either Hypotheses \ref{hyp2} with $m=0$ or Hypotheses \ref{dissiR} with $\zeta_R<0$.
For $p\geq 1$ and $\varphi\in C^1_b(H)$, the following inequality holds:
\begin{align}\label{logsob}
\int_H\abs{\varphi}^p\ln\abs{\varphi}^pd\nu-\pa{\int_H|\varphi|^pd\nu}&\ln\pa{\int_H|\varphi|^pd\nu}\leq p^2C\int_H\abs{\varphi}^{p-2}\norm{\nabla_R \varphi}_R^2\chi_{\set{\varphi\neq 0}}d\nu
\end{align}
where
\begin{align}\label{Costante_C}
C:=\|\psi\|_{L^1((0,\infty))}=\eqsys{K\pa{\frac{r(1-\gamma,\min\{\omega,|\zeta|\})}{\min\{\omega,|\zeta|\}^{1-\gamma}}+\frac{e^{\min\{\omega,|\zeta|\}}}{\min\{\omega,|\zeta|\}}}, & \text{Hypotheses \ref{hyp2} with $m=0$,}\\
\frac{1}{2|\zeta_R|}, & \text{Hypotheses \ref{dissiR} with $\zeta_R<0$;}
}
\end{align}
where $K$ is introduced in \eqref{Stima_sporchina} and $r$ is the lower incomplete Gamma function, i.e. $r(s,x):=\int_0^x \xi^{s-1}e^{-\xi}d\xi$, for $x>0$ and $s\in \mathbb{C}$ with positive real part.
Further, if $\varphi \in W^{1,2}_R(H, \nu)$ then
\begin{align}\label{logsob2}
\int_H\abs{\varphi}^2\ln\abs{\varphi}^2d\nu-\pa{\int_H|\varphi|^2d\nu}&\ln\pa{\int_H|\varphi|^2d\nu}\leq 4C\int_H\norm{\nabla_R \varphi}_R^2\chi_{\set{\varphi\neq 0}}d\nu.
\end{align}
\end{thm}

\begin{proof}
The proof is quite similar to that provided in \cite[Theorem 1.9]{BF22} and we provide it for the sake of completeness. However, for the sake of completeness, we give a sketch of it.

We begin by establishing \eqref{logsob} for a function $\varphi$ in the space $\mathcal{F}C_b^1(H)$, provided that there exists a positive constant $c$ satisfying the condition $c\leq \varphi\leq 1$. To this aim we consider the function
\begin{equation*}
[0,\infty) \ni t \mapsto G(t):=\int_H (P(t)\varphi^p)\ln (P(t)\varphi^p)d\nu,
\end{equation*}
which is well defined thanks to the contractivity and the positivity preserving of $P(t)$. 
By a standard argument involving the invariance of $\nu$, \eqref{Azala} and the fact that $P(t)\norm{\nabla_R \varphi^p}_R\leq [P(t)(\norm{\nabla_R \varphi^p}^2_R \varphi^{-p})]^{1/2}\pa{P(t)\varphi^p}^{1/2}$ we have
\begin{align*}
G'(t)&\geq -\psi(t)p^2\int_{H} \varphi^{p-2}\norm{\nabla_R \varphi}^2_R d\nu.
\end{align*}
Integrating the latter inequality w.r. to $t$ from $0$ to $+\infty$ and by taking \eqref{trans1} into account, we get
\begin{gather*}
\int_{H}\varphi^p\ln \varphi^pd\nu- \pa{\int_{H}\varphi^pd\nu}\ln\pa{\int_{H}\varphi^pd\nu}\leq p^2\|\psi\|_{L^1((0,+\infty))}\int_{H}\varphi^{p-2}\norm{\nabla_R \varphi}_R^2d\nu.
\end{gather*}
The general case follows by standard approximation arguments and the Fatou lemma.
\end{proof}

A classical consequence of the logarithmic Sobolev inequality \eqref{logsob2} is the Poincaré inequality (see for instance \cite[Theorem 5.2]{AngLorLun} for a proof in finite dimension and \cite[Theorem 1.10]{BF22} for a proof in infinite dimension).
\begin{cor}
Assume either Hypotheses \ref{hyp2} with $m=0$ or Hypotheses \ref{dissiR} with $\zeta_R<0$. Then, for any $\varphi \in W^{1,2}_R(H,\nu)$
\begin{align*}
\|\varphi -m_{\nu}(\varphi)\|_R^2\le C \|\nabla_R\varphi\|_R^2,
\end{align*}
where $C$ is the same constant in \eqref{Costante_C} and $m_{\nu}(\varphi)$ denotes the average of $\varphi$ with respect to $\nu$, i.e.,
\[m_{\nu}(\varphi):= \int_H \varphi d\nu, \qquad \varphi \in L^1(H, \nu).\]
\end{cor}

The formulation of the logarithmic Sobolev inequality given by Gross in \cite{Gross93} is
\begin{equation}\label{lsgross}
\int_H |f|^p\ln |f|d\sigma-\|f\|^p_{L^p(H,\sigma)}\ln\|f\|_{L^p(H,\sigma)}\le c \int_H f|f|^{p-2}(\mathcal L f)d\sigma
\end{equation}
where $c$ is a positive constant and $\mathcal L$ is the infinitesimal generator of a strongly continuous semigroup $T(t)$ in $L^p(H,\sigma)$, for a suitable choice of a probability measure $\sigma$.
In our case $\sigma=\nu$, $\mathcal{L}= \mathcal{N}_0$, $T(t)=P(t)$ and by \eqref{Hazel} the inequalities \eqref{logsob} and \eqref{lsgross} coincide.  Therefore, applying \cite[Theorems 3.7]{Gross93} we deduce the hypercontractivity of the semigroup $P(t)$.
that means that for any fixed $p,q \in (1, +\infty)$ with $p<q$, there exists $t_0=t_0(p,q)>0$ such that $P(t)$ maps $L^p(H,\nu)$ into $L^q(H, \nu)$ for any $t \ge t_0$ and
\begin{align*}
\|P(t)\|_{\mathcal{L}(L^p(H,\nu);L^q(H, \nu))}\le 1, \qquad t \ge t_0.
\end{align*}

\begin{thm}\label{Hyper}
Assume either Hypotheses \ref{hyp2} with $m=0$ or Hypotheses \ref{dissiR} with $\zeta_R<0$. For every $t>0$, $q\in (1,+\infty)$ and $p\leq (q-1)e^{t/(2C)}+1$ it holds
\begin{align}\label{iper_not_intro}
\norm{P(t) \varphi}_{L^p(H,\nu)}\leq \norm{\varphi}_{L^q(H, \nu)},\qquad \varphi\in L^q(H, \nu).
\end{align}
where $C$ is defined in \eqref{Costante_C}.
\end{thm}
Thanks to \cite[Theorem 3.12]{Gross93}, one can show that the logarithmic Sobolev inequality \eqref{logsob} and the hypercontractivity of $P(t)$ are equivalent.

\begin{rmk}
If $A=-\Id_H$ and $F\equiv 0$ in \eqref{SPDE} then Hypotheses \ref{dissiR} is verified with $\zeta_R=-1$ and so the constant $C$ in \eqref{logsob} is equal to $1/2$. Moreover estimate \eqref{iper_not_intro} holds for every $t>0$ and $p\leq C(t,q):=(q-1)e^{-t}+1$. We stress that, in this framework, $C(t,q)$ is optimal as explained in \cite[Remark 3.4]{Gross75}.
\end{rmk}

A less standard result concerns the exponential integrability of Lipschitz functions along $H_R$ with respect to $\nu$. 

\begin{thm}\label{code}
Assume either Hypotheses \ref{hyp2} with $m=0$ or Hypotheses \ref{dissiR} with $\zeta_R<0$
and that $\nu(H_R)=1$. Any function $g\in {\rm{Lip}}_{H_R}(H)$ with ${\rm{Lip}}_R(g)\le 1$
belongs to $L^1(H, \nu)$ and for any $t>0$ it holds
\begin{align*}
\nu(\{x\in H\,|\, g(x) \ge m_\nu(g)+t\})\le e^{-\frac{t^2}{16\sqrt{2}C}}
\end{align*}
where $C$ is defined in \eqref{Costante_C}.
Further, if $\alpha<(16\sqrt{2}C)^{-1}$
\begin{equation}\label{McGuyver}
\int_{H_R} e^{\alpha g^2} d\nu <\infty.
\end{equation}
\end{thm}

\begin{proof}
First of all, let us observe that \eqref{logsob} can be easily proved for any $\varphi\in C^1_{b,H_R}(H)$, i.e. the space of bounded and $H_R$-differentiable functions from $H$ to $\R$.
We claim that
\begin{align}\label{Roger}
\tau\int_H g e^{\tau g}d\nu - m_\nu(e^{\tau g})&\ln\left( m_\nu(e^{\tau g})\right)\leq  4\sqrt{2}C\tau^2 m_\nu(e^{\tau g}).
\end{align}
for any $\tau>0$ and any bounded $g \in\textrm{Lip}_{H_R}(H)$ with ${\rm{Lip}}_R(g)\le 1$.
To this aim, let us fix $\tau$ and $g$ as above and let $g_\eps$ be the Lasry--Lions approximations along $H_R$ of $g$ (see Appendix \ref{LASRY}). Since $g_\eps$ belongs to $C^1_{b,H_R}(H)$ and $\norm{\nabla_R g_\eps}_R\le 4\sqrt{2}{\rm Lip}_{R}(g)\le  4\sqrt{2}$, by \eqref{logsob} with $\varphi=e^{\tau g_\eps}$ and $p=1$ we get
\begin{align}\label{Volcano}
\tau\int_H g_\eps e^{\tau g_\eps}d\nu - m_\nu(e^{\tau g_\eps})&\ln\left( m_\nu(e^{\tau g_\eps})\right)\leq C\tau^2\int_He^{\tau g_\eps}\norm{\nabla_R g_\eps}_R^2d\nu\leq  4\sqrt{2}C\tau^2 m_\nu(e^{\tau g_\eps}).
\end{align}
Since $g_\eps$ converges pointwise to $g$ and $|g_\eps|\le \|g\|_\infty$ (Proposition \ref{Prop_LASRY}), by \eqref{Volcano} and the dominated convergence theorem we get the claim.
Now, introducing
\begin{align*}
H(\tau):=\eqsys{m_\nu(g), & \tau=0;\\ \tau^{-1}\ln \pa{m_\nu(e^{\tau g})}, & \tau>0,}
\end{align*}
by \eqref{Roger} we get $H'(\tau)\leq 4\sqrt{2}C$, whence integrating from $0$ to $t$ we obtain
\begin{align}\label{nonna}
H(t)-H(0)\leq 4\sqrt{2}Ct\ \Longleftrightarrow\ m_\nu(e^{t g})\leq e^{4\sqrt{2}Ct^2+t m_\nu(g)}.
\end{align}
The Chernoff bound (see \cite[Introduction of Section 4.2]{MU05}) and \eqref{nonna} yield
\begin{align}\label{G_tail}
\nu(\{x\in H\,|\, g(x)\geq m_\nu(g)+s\})&\leq \inf_{t\geq 0}\pa{e^{-t(s+m_\nu(g))}m_\nu(e^{tg})}\notag\\
&\leq\inf_{t\geq 0}\pa{e^{-ts+4\sqrt{2}Ct^2}}=e^{-\frac{s^2}{16\sqrt{2} C}}.
\end{align}

We now consider the general case. Let $g$ be as in the statement and let $\psi_n\in C^1_b(\mathbb{R})$ be an odd increasing function such that
$\psi_n(t)=t$ for $t \in [0,n]$, $\psi_n(t)=n+1$ if $t \ge n+2$ and $\psi'_n(t) \in [0,1]$ for any $t \in \mathbb{R}$.
Then, the function $g_n=\psi_n\circ g$ belongs to $\textrm{Lip}_{H_R}(H)$, satisfies ${\rm{Lip}}_R(g_n)\le 1$ and is bounded as well. By the previous step, applying \eqref{G_tail} also to $-|g_n|$ we deduce
$$
\nu \left(\{x\in H\,|\, |g_n(x)|\le m_\nu(|g_n|)-t\}\right)\le
e^{-\frac{t^2}{16\sqrt{2}C}}.
$$
Choosing $t_0>0$ such that $e^{-t_0^2/(16\sqrt{2}C)}\le 1/2$ and $m$ such that
$\nu (\{x\in H\,|\, |g(x)|\ge m\})<1/2$ and using that $|g_n|\le |g|$ we deduce that
$\|g_n\|_{L^1(H,\nu)}=m_\nu(|g_n|)\le m+t_0$ for any $n \in \N$. Indeed, by contradiction,
if  $m_\nu(|g_n|)>m+t_0$, we have
\begin{align*}
\nu(\{x\in H\,|\,|g(x)|<m\})&\le \nu(\{x\in H\,|\,|g_n(x)|<m\})\\
& = \nu(\{x\in H\,|\,|g_n(x)|+t_0<m+t_0\})
\\
&\le \nu(\{x\in H\,|\,|g_n(x)|+t_0<m_\nu(|g_n|)\})\\
&= \nu(\{x\in H\,|\,|g_n(x)|<m_\nu(|g_n|)-t_0\})\le 1/2
\end{align*}
which yields a contradiction with $\nu (\{x\in H\,|\,|g(x)|\ge m\})<1/2$, as $\nu$ is a probability measure.
Hence, by the previous estimate  and the monotone convergence theorem we get that
$g\in L^p(H, \nu)$ for any $p \ge 1$ and that $\|g_n\|_{L^p(H,\nu)}$ converges to
$\|g\|_{L^p(H,\nu)}$ as $n$ approaches infinity. Moreover, using \eqref{G_tail} and the fact that
$m_\nu(|g_n|)\le m+t_0$ we obtain
\begin{equation*}
\nu(\{x\in H\,|\,|g_n(x)|\ge m+t_0+t\})\le e^{-\frac{t^2}{16\sqrt{2}C}}\le e^{-\frac{t_0^2}{16\sqrt{2}C}}
\end{equation*}
for $t\ge t_0$ whence $\sup_{n\in\N}\|g_n\|_{L^2(H, \nu)}<+\infty$. By a standard compactness argument
we get that $g_n$ converges to $g$ in $L^2(H, \nu)$ as $n \to +\infty$ and hence in measure,
i.e., for any $\varepsilon>0$
\begin{equation}\label{measure}
\lim_{n \to +\infty}\nu(\{x\in H\,|\,|g_n(x)-g(x)|\ge \varepsilon\})=0.
\end{equation}
From \eqref{G_tail} and \eqref{measure} for $g_n$ we infer that \eqref{G_tail} holds true for $g$ as well.

We now prove \eqref{McGuyver}. First of all, let us observe that for any $\alpha>0$, the layer cake formula implies
\begin{align*}
\int_H e^{\alpha(g(x))^2} \nu(dx) & \le 1+\int_1^\infty \nu \pa{\set{x\in H\,\middle |\, e^{\alpha(g(x))^2}>s}}ds
\\
&=1+\int_1^\infty \nu \pa{\set{x\in H\,\middle|\,(g(x))^2>\frac{\ln s}{\alpha}}}ds.
\end{align*}
Now, performing the change of variables $s=e^{\alpha(t+m_\nu(g))^2}$ we get
\begin{align*}
\int_H e^{\alpha(g(x))^2} \nu(dx) & \le 1+2\alpha\int_{\R} \nu\{g>m_{\nu}(g)+t\} e^{\alpha(m_{\nu}(g)+t)^2}(m_{\nu}(g)+t) dt.
\end{align*}
Thus, using estimate \eqref{G_tail} and choosing $\alpha<(16\sqrt{2}C)^{-1}$ we conclude the proof.
\end{proof}

The following is a Fernique-type result for the measure $\nu$.

\begin{cor}\label{var_fer}
Assume either Hypotheses \ref{hyp2} with $m=0$ or Hypotheses \ref{dissiR} with $\zeta_R<0$
and that $\nu(H_R)=1$. The map $x \mapsto e^{\lambda \|x\|_R^2}$ belongs to $L^1(H_R, \nu)$ for any $\lambda<(16\sqrt{2}C)^{-1}$ where $C$ is defined in \eqref{Costante_C}.
\end{cor}

 Using the fact that $\nabla_R\varphi=R^2\nabla\varphi$ for every $\varphi\in C^1_b(H)$ (see Proposition \ref{dalpha}), the proof of Theorem \ref{code} can be repeated similarly if we replace a $H_R$-Lipschitz function with a Lipschitz continuous function without assuming that $\nu(H_R)=1$.  This fact allow us to obtain the following results.

\begin{thm}\label{codeH}
Assume either Hypotheses \ref{hyp2} with $m=0$ or Hypotheses \ref{dissiR} with $\zeta_R<0$. Any function $g\in {\rm{Lip}}(H)$ with ${\rm{Lip}}(g)\le 1$
belongs to $L^1(H, \nu)$ and for any $t>0$ it holds
\begin{align*}
\nu(\{x\in H\,|\, g(x) \ge m_\nu(g)+t\})\le e^{-\frac{t^2}{16\sqrt{2}C'}}
\end{align*}
where $C'=\norm{R}_{\mathcal{L}(H)}C$ and $C$ is defined in \eqref{Costante_C}. Further, if $\alpha<(16\sqrt{2}C)^{-1}$
\begin{equation}
\int_{H} e^{\alpha g^2} d\nu <\infty.
\end{equation}
and in particular $x \mapsto e^{\lambda \|x\|^2}$ belongs to $L^1(H, \nu)$ for any $\lambda<(16\sqrt{2}C')^{-1}$.
\end{thm}

\section{Supercontractivity and related results}\label{SUPER}

In this section we are interested in proving some characterizations for the semigroup $P(t)$ to be supercontractive, i.e. for any $1<p<q<+\infty$ and $t>0$
\begin{equation*}
\|P(t)\|_{\mathcal{L}(L^p(H, \nu); L^q(H, \nu))}\le c_{p,q}(t),
\end{equation*}
for some function $c_{p,q}:(0,+\infty)\to (0,+\infty)$ that blows up as $t\to 0^+$. Analogous results are known for semigroups associated to uniformly elliptic operators with unbounded coefficients in finite dimension, see for example \cite{AngLor}.
In this section we assume Hypotheses \ref{dissiR} restricting our attention to the dissipative case. Indeed, as it is well known already in finite dimension, the Ornstein--Uhlenbeck semigroup associated to the one dimensional operator
\[(L\varphi)(\xi)=\Delta \varphi(\xi)-\xi\varphi'(\xi),\qquad \xi \in \mathbb{R},\ \varphi\in C^2_b(\R),\]
is not supercontractive with respect to its invariant measure, namely the Gaussian measure $\mu(dx)= (2\pi)^{-1/2}e^{-x^2/2}dx$, as proved in \cite{nelson}.

Finite dimensional results cannot be directly used in the infinite dimensional case. However, by employing a finite approximation procedure alongside Harnack inequalities, we will obtain the essential tools needed to derive our results in the infinite dimensional case.

We start by recalling the Harnack inequality proved in \cite{ABF21}
\begin{align}\label{dae}
|(P(t)\varphi)(x+h)|^p\leq (P(t)|\varphi|^p)(x)\cdot\exp\left(\frac{p(e^{2\zeta_R t}-1)}{2\zeta_R(p-1) t^2}\|h\|_R^2\right),
\end{align}
which holds true for any $\varphi\in B_b(H)$,  $t>0$, $x\in H$, $h\in H_R$ and $p>1$.

To begin with, we first construct the finite dimensional sequence of semigroups which approximates $P(t)$.
We start by recalling the definition of the Yosida approximation procedure for the function $F$. For any $\delta>0$ and $x\in H$,  we let
$J_\delta(x)\in \Dom(F)$ be the unique solution of
\[y-\delta (F(y)-\zeta_F y)=x,\]
where $\zeta_F$ was introduced in Hypothesis \ref{main}(iv).
The existence of $J_\delta(x)$, for every $x\in H$ and $\delta>0$, is guaranteed by \cite[Proposition 5.3.3]{DA-ZA2}. We define $F_\delta:H\ra H$ as
\begin{align*}
F_\delta(x):=F(J_\delta(x)),\qquad x\in H,\ \delta>0.
\end{align*}
Now, we set
\begin{equation*}
\overline{\zeta}_F:=\left\{\begin{array}{ll}|\zeta_F|^{-1} \qquad\,\,&{\rm if}\,\,\zeta_F\neq 0\\
+\infty\qquad\;\,&{\rm if}\,\,\zeta_F=0
\end{array}\right.
\end{equation*}
We recall that for any $\delta \in (0,\overline{\zeta}_F)$ and $\mathcal{O}\in\{H,E,H_R\}$, the function $F_\delta:\mathcal{O}\ra\mathcal{O}$ is Lipschitz continuous and the function $F_\delta-\zeta_F\Id_{\mathcal{O}}:\mathcal{O}\ra\mathcal{O}$ is dissipative and satisfies Hypothesis \ref{main}(v) Hence by Theorem \ref{Genmild} for every $\delta \in (0,\overline{\zeta}_F)$ and $x\in H$, the problem
\begin{gather*}
\eqsys{
dX_\delta(t)=\big[AX_\delta(t)+F_\delta(X_\delta(t))\big]dt+RdW(t), & t>0;\\
X_\delta(0)=x,
}
\end{gather*}
has a unique mild solution $\{X_\delta(t,x)\}_{t\geq 0}$. Consequently the semigroup
\[
(P_\delta(t)\varphi)(x):=\E[\varphi(X_\delta(t,x))]
\]
is well defined for any $\varphi \in B_b(H)$. The proof of the following proposition can be found in \cite[Proposition 4.3]{ABF21}.

\begin{prop}
Assume Hypotheses \ref{dissiR} hold true. For any $T>0$, $\varphi\in C_b(H)$ and $x\in E$
\begin{align}
\lim_{\delta\ra 0}\sup_{t\in [0,T]}\norm{X_\delta(t,x)-X(t,x)}_E&=0,\qquad \mathbb{P}\text{-a.s.};\notag\\
\lim_{\delta\ra 0}\abs{(P_\delta(t)\varphi)(x)-(P(t)\varphi)(x)}&=0,\qquad t>0.\label{convE1}
\end{align}
\end{prop}

Now we introduce the finite dimensional approximation procedure we will use throughout this section. For any $n\in\N$ we define
\[H_n:={\rm span}\{e_1,\dots, e_n\},\]
where $\{e_k\,|\, k \in \N\}$ is an orthonormal basis of $H$.
Further, let $\pi_n:H\ra H_n$ be the orthogonal projection in $H$. For any $n \in \N$ and $\delta \in (0,\overline{\zeta}_F)$ we define
$A_n:H\ra H_n$, $R_n:H\ra H_n$ and $F_{\delta,n}:H\ra H_n$ by
\begin{align*}
A_n:=\pi_nA\pi_n(=A \pi_n) ,\qquad R_n:=\pi_nR\pi_n(=R \pi_n) \qquad\text{ and }\qquad F_{\delta,n}:=\pi_nF_{\delta} \pi_n
\end{align*}
Now, fixed $n \in \N$ and $\delta \in (0,\overline{\zeta}_F)$ we consider
\begin{equation}\label{SDE_n}
\left\{\begin{array}{ll}
dX_n(t)=[A_nX_n(t)+F_{\delta,n}(X_n(t))]dt+ R_ndW_n(t),&  t> 0,\\
X_n(0)=x\in H_n.
\end{array}\right.
\end{equation}
Here $W_n(t):=\pi_nW(t)= \sum_{k=1}^n \langle W(t),e_k\rangle_H e_k$.

It is straightforward to check that $A_n, R_n$ and $F_{\delta,n}$ satisfy Hypotheses \ref{main}. Moreover, being $R$ an injective operator it follows that $R_n$ is bijective, hence $RH_n=H_n$ for any $n \in \N$. Moreover, fixed $\delta \in (0,\overline{\zeta}_F)$, $n \in \N$ and $x \in H_n$, by Theorem \ref{Genmild} we can deduce existence and uniqueness of a mild solution $\{X_{\delta,n}(t,x)\}_{t\geq 0}$ of \eqref{SDE_n} and consequently well-posedness for the associated transition semigroup defined for $f\in B_b(H_n)$ as
\begin{align}\label{Sem_fin}
(P_{\delta,n}(t)f)(x):=\mathbb{E}[f(X_{\delta,n}(t,x))],\qquad t>0,\ x\in H_n.
\end{align}
We recall that for any $n\in\N$, $\varphi\in C^2_b(H_n)$ and $x\in H_n$ we have
\[
\lim_{t\ra 0}\dfrac{(P_{\delta,n}(t)\varphi)(x)-\varphi(x)}{t}=\mathcal{N}_{\delta,n}\varphi(x),
\]
where
\begin{equation}\label{OUP}
\mathcal{N}_{\delta,n} \varphi(x)= \frac{1}{2}{\rm Tr}[R^2_n D^2 \varphi(x)]+ \langle A_n x+ F_{\delta,n}(x), \nabla\varphi(x)\rangle_H,
\end{equation}
see \cite[Section 4.1]{BI1} or \cite{GOKO01}. The proof of the following proposition can be found in \cite[Proposition 3.3]{ABF21}.

\begin{prop}
Assume Hypotheses \ref{dissiR} hold true. For any $\delta \in (0,\overline{\zeta}_F)$, $f\in C_b(H)$, $t\geq 0$ and $x\in H_{n_0}$, for some $n_0\in\N$, it holds
\begin{align}\label{fin_approx}
\lim_{n\ra+\infty}P_{\delta,n}(t)f(x)=P_\delta(t)f(x).
\end{align}
\end{prop}

In the next result we show how estimate \eqref{dae} plays a key role to prove a characterization of the supercontractivity of $P(t)$ in terms of some functional inequalities and some exponential integrability conditions (see Theorem \ref{sup_thm} below).

\begin{prop}
Assume Hypotheses \ref{dissiR} hold true with $\zeta_R<0$. For any $t>0$, $f\in C^{1}_b(H)$ and $x\in H$ it holds
\begin{equation}\label{prelog}
P(t)(\abs{f}^2\ln\abs{f}^2)(x)\leq P(t)(\abs{f}^2)(x)\ln P(t)(\abs{f}^2))(x) +C(t)P(t)(\norm{\D_R f}^2_R)(x).
\end{equation}
Here $C:(0,+\infty)\ra(0,+\infty)$ is defined as $C(t):=3|\zeta_R|^{-1}(1-e^{2\zeta_R t})$.
\end{prop}

\begin{proof}
The proof of estimate \eqref{prelog} follows the same line of that of \cite[Proposition 3.2]{AngLor}. However, for the sake of completeness we give a sketch of it reducing ourselves to the case when $f\in C^1_b(H)$ has positive infimum and $\sup_{x\in H} f(x) \le 1$. This latter condition can be removed by applying \eqref{prelog} to the function $f\|f\|_\infty^{-1}$ observing that $P(t)c=c$ for any $t>0$ and $c \in \R$. The general case can be obtained by the dominated convergence theorem writing \eqref{prelog} replacing $f$ with $(f^2+m^{-1})^{1/2}$ and letting $m \to +\infty$.

Fix $n\in\N$ and $\delta \in (0,\overline{\zeta}_F)$ and let us consider the semigroup $\{P_{\delta,n}(t)\}_{t\geq 0}$ given by \eqref{Sem_fin} and a nonnegative function $f\in C^{1}_b(H_n)$ such that $\sup_{x\in H_n}f(x)\leq 1$ and $\inf_{x\in H_n}f(x)>\eps$ for some $\eps>0$.
For any fixed $t>0$ we consider the function
\begin{equation}\label{gn}
g_{\delta,n}(r,x):=(P_{\delta,n}(r)f)(x) \qquad x\in H_n,\;r\in [0,t].
\end{equation}
The function $g_{\delta,n}:[0,T]\times H_{n}\ra \R$ belongs to $C^{1,2}([0,T]\times H_{n})$ and solves
\begin{gather*}
\eqsys{
\frac{d}{dr}g_{\delta,n}(r,x)=\mathcal{N}_{\delta,n}g_{\delta,n}(r,x), & r\in (0,t);\\
g_{\delta,n}(0,x)=f(x),
}\end{gather*}
where $\mathcal{N}_{\delta,n}$ is the operator defined in \eqref{OUP}. Moreover
$g_{\delta,n}(r,x)\geq \eps$ for any $r\geq 0$ and $x\in H_{n}$ (see \cite[Theorem 1.2.5]{lorbook}) and $\sup_{x\in H_n}g_{\delta,n}(r,x)\leq 1$, for any $r\in [0,t]$. For simplicity, we set $g(r):=g_{\delta,n}(r,x)$ and we consider the function
\[
G(r):=P_{\delta,n}(t-r)(g(r)^2\ln g(r)^2)(x),\qquad r\in [0,t].
\]
It is not difficult to show that
\begin{align}\label{G'n}
G'(r)=P_{\delta,n}(t-r)\left[-\mathcal{N}_{\delta,n}(g(r)^2\ln g(r)^2)+2g(r)g'(r)(1+\ln g(r)^2)\right](x).
\end{align}
By \eqref{OUP} and \eqref{gn} we have $g'(r)=\mathcal{N}_{\delta,n}g(r)$ and
\[
\mathcal{N}_{\delta,n}(g^2\ln g^2)=2g(1+\ln g^2)\mathcal{N}_{\delta,n}g+2(3+\ln g^2)\norm{R_n\D_x g}^2,
\]
hence for any $x\in H_n$ and $r\in [0,t]$ \eqref{G'n} becomes
\[
G'(r)=-P_{\delta,n}(t-r)\left[2(3+\ln (g(r))^2)\norm{R_n\D_x g(r)}^2\right](x).
\]
Recalling that $\eps<\inf_{x\in H_n}g(x)\leq \sup_{x\in H_n}g(x)\leq 1$ we obtain
\[
G'(r)\geq-6P_{\delta,n}(t-r)\left[\norm{R_n\D_x P_{\delta,n}(r) f}^2\right](x).
\]
Applying estimate \eqref{Azala} to $P_{\delta,n}(t)$ we deduce that
\[
G'(r)\geq-6e^{2\zeta_R r}P_{\delta,n}(t)\left(\norm{R_n\D_x f}^2\right)(x)
\]
whence, integrating with respect to $r$ from $0$ and $t$, we get
\begin{align*}
G(t)-G(0)=\int^t_0G'(r)dr &\geq -6\left(\int^t_0e^{2\zeta_R r}dr\right)P_{\delta,n}(t)\left(\norm{R_n\D_x f}^2\right)(x)\\
&\geq -3\zeta_R^{-1}(e^{2\zeta_R t}-1)P_{\delta,n}(t)\left(\norm{R_n\D_x f}^2\right)(x),
\end{align*}
or, equivalently,
\begin{equation}\label{first}
P_{\delta,n}(t)(f^2\ln f^2)(x)\leq C(t)P_{\delta,n}(t)(\norm{R_n\D f}^2)(x)+P_{\delta,n}(t)(f^2)(x)\ln P_{\delta,n}(t)( f^2)(x)
\end{equation}
for any $t>0$, where $C(t)=3|\zeta_R|^{-1}(1-e^{2\zeta_R t})$. Letting $n\ra +\infty$ in \eqref{first} and taking into account \eqref{fin_approx} and Proposition \ref{dalpha} we obtain
\begin{equation}\label{prelogdelta}
P_\delta(t)(f^2\ln f^2)(x)\leq C(t)P_\delta(t)(\norm{\D_R f}_R^2)(x)+P_\delta(t)(f^2)(x)\ln P_\delta(t)( f^2)(x),
\end{equation}
for any $\delta \in (0,\overline{\zeta}_F)$, $x\in \bigcup_{n\in\N}H_n$ and $f$ as above. Moreover by the continuity of $f$ and $\D_R f$ and by the density of $\bigcup_{n\in\N}H_n$ in $H$, the inequality \eqref{prelogdelta} hold true for any $x\in H$.
Again for the continuity of $f$ and $\D f$ by \eqref{convE1} letting $\delta\ra 0$ we obtain \eqref{prelog} for any $x\in E$. Finally, the density of $E$ in $H$, property \eqref{gmild} and the dominated convergence theorem imply \eqref{prelog} for any $x\in H$.
\end{proof}

\begin{rmk}{\rm We point out that if, for any $t>0$ and $x\in H$, $\nu_{t,x}$ denotes the law of the random variable $X(t,x)$, then the inequality \eqref{prelog} can be reformulated as follows
\begin{align}
\int_H\abs{f}^2\ln\abs{f}^2 d\nu_{t,x}&\leq \left(\int_H\abs{f}^2d\nu_{t,x}\right)\ln \left(\int_H\abs{f}^2d\nu_{t,x}\right)+ C(t)\int_H\norm{\D_R f}^2_Rd\nu_{t,x}.\label{prelog-var}
\end{align}
Observe that \eqref{prelog-var} is a logarithmic Sobolev inequality for the probability measure $\nu_{t,x}$.
}
\end{rmk}

To get a characterization for the supercontractivity of $P(t)$ we need to introduce a family of logarithmic Sobolev inequalities, called $\varepsilon$-logarithmic Sobolev inequalities that have the form 

\begin{align}\label{epslog}
\int_H f^2\ln|f|d\nu-\pa{\int_H f^2d\nu}\ln\|f\|_{L^2(H,\nu)}&\leq  \eps\int_H\norm{\D_R f}_R^2d\nu+\beta(\eps)\norm{f}_{L^2(H,\nu)}^2,
\end{align}
for any $f\in W_R^{1,2}(H,\nu)$, $\eps>0$  and some decreasing function $\beta :(0,+\infty)\ra\R$ blowing up as $\eps\ra 0^+$.

\begin{prop}\label{code1}
Assume Hypotheses \ref{dissiR} hold true. If $\nu(H_R)=1$ and inequality \eqref{epslog} is satisfied, then for any $g \in {\rm{Lip}}_{H_R}(H)$ with ${\rm Lip}_R(g)\le 1$ and $t>0$ it holds
\begin{align*}
\nu(\{x\in H\,|\,g(x) \ge m_\nu(g)+t\})\le \inf_{\varepsilon>0}\exp\left(-\frac{1}{2\varepsilon}[t+m_\nu(g)-\ln K-2\beta(\varepsilon)]^2-2\beta(\varepsilon)\right),
\end{align*}
for some positive constant $K$. Consequently, for any $\alpha>0$
\begin{align*}
\int_{H_R} e^{\alpha g^2} d\nu <+\infty.
\end{align*}
\end{prop}

\begin{proof}
Let $g$ be as in the statement and let $g_n$ be the sequence of functions introduced in the proof of Theorem \ref{code}.
Set
\[
K_{n}(r):=\int_{H_R}e^{ r g_n}d\nu,\qquad n \in \N,\ r>0.
\]
Using estimate \eqref{McGuyver} and the inequality $|\psi_n(t)|\le |t|$ which holds true for any $t\in \R$, $n \in \N$,  we deduce
\begin{equation}\label{N0}
 K_n(r)\leq K:=\int_{H_R}e^{|g|}d\nu<+\infty,\qquad r\leq 1.
\end{equation}
Now, applying \eqref{epslog} to the function $e^{\frac{r}{2} g_n}$  we obtain
\begin{equation*}
rK'_{n}(r)-K_n(r)\ln K_n(r)\leq \frac{\eps}{2}r^2K_n(r)+2\beta(\eps)K_n(r)
\end{equation*}
dividing by $r^2K_n(r)$ we have
\begin{equation}\label{N2}
\left( \frac{1}{r}\ln K_n(r) \right)'\leq \frac{\eps}{2}+\frac{2}{r^2}\beta(\eps).
\end{equation}
 Integrating \eqref{N2} from $1$ to $v>1$ we get
\begin{equation*}
 \frac{1}{v}\ln K_n(v)- \ln K_n(1)\leq \frac{\eps}{2}(v-1)+2\beta(\eps)\left(1-\frac{1}{v}\right),
\end{equation*}
and so by \eqref{N0}, for any $n\in\N$, $v>1$ and $\eps>0$ we have
\begin{equation}\label{N3}
 K_n(v)=\int_{H_R}e^{ v g_n}d\nu\leq \exp\pa{\frac{\eps}{2}v^2+v\left[\ln K+2\beta(\eps)\right]-2\beta(\eps)}.
\end{equation}
Now, by the Chernoff bound (see \cite[Introduction of Section 4.2]{MU05}) and \eqref{N3} we get
\begin{align}\label{nonG_tail}
\nu(\{x\in H\,|\, g_n(x)\geq m_\nu(g_n)+s\})&\leq \inf_{t\geq 0}\pa{e^{-t(s+m_\nu(g_n))}m_\nu(e^{tg_n})}\notag\\
&\leq\inf_{t\geq 1}\pa{\exp\pa{-t[s+m_\nu(g_n)]+\frac{\eps}{2}t^2+t\left[\ln K+2\beta(\eps)\right]-2\beta(\eps)}}\notag\\
&=\exp\pa{-\frac{1}{2\eps}[s+m_\nu(g_n)-\ln K-2\beta(\eps)]^2-2\beta(\eps)}.
\end{align}
Using a similar argument as the one used in the proof of Theorem \ref{code} we obtain that \eqref{nonG_tail} can be deduced for the function $g$ too.
By the Layer cake formula, for any $\alpha>0$, we get
\begin{align*}
\int_H e^{\alpha g^2} d\nu & \le 1+\int_1^\infty \nu \pa{\set{x\in H\,\middle |\, e^{\alpha(g(x))^2}>s}}ds
\\
&=1+\int_1^\infty \nu \pa{\set{x\in H\,\middle|\,|g(x)|>\sqrt{\frac{\ln s}{\alpha}}}}ds.
\end{align*}
Performing the change of variables $s=e^{c(t+m_\nu(|g|))^2}$ we obtain
\begin{align*}
\int_H e^{\alpha g^2} d\nu & \le 1+2\alpha\int_{\R} \nu\{|g|>m_{\nu}(|g|)+t\} e^{\alpha(m_{\nu}(|g|)+t)^2} (t+m_\nu(|g|))dt
\end{align*}
 and using estimate \eqref{nonG_tail} with $f_n$ being replaced by $|g|$ it holds
  \begin{align*}
&\int_H e^{\alpha g^2} d\nu \\
& \le 1+2\alpha\int_{\R}\exp\pa{-\frac{1}{2\eps}\left[t+m_\nu(|g|)-\ln K-2\beta(\eps)\right]^2-2\beta(\eps)+\alpha(m_{\nu}(|g|)+t)^2} (t+m_\nu(|g|))dt,
\end{align*}
whence, choosing $\varepsilon< (2\alpha)^{-1}$ we can conclude.
\end{proof}

As in Theorem \ref{code}, using the relation $\nabla_R\varphi=R^2\nabla\varphi$ in \eqref{epslog} we obtain the following alternative result.
\begin{prop}\label{code1H}
Assume Hypotheses \ref{dissiR} hold true. If inequality \eqref{epslog} is satisfied, then for any $g \in {\rm{Lip}}(H)$ with ${\rm Lip}(g)\le 1$ and $t>0$ it holds
\begin{align*}
\nu(\{x\in H\,|\,g(x) \ge m_\nu(g)+t\})\le \inf_{\varepsilon>0}\exp\left(-\frac{1}{2\varepsilon}[t+m_\nu(g)-\ln K-2\beta(\varepsilon)]^2-2\beta(\varepsilon)\right),
\end{align*}
for some positive constant $K$. Consequently, for any $\alpha>0$
\begin{align*}
\int_{H} e^{\alpha g^2} d\nu <+\infty.
\end{align*}
\end{prop}
We are going to prove that the supercontractivity property of $P(t)$ implies \eqref{epslog}.
\begin{thm}\label{super-eps}
Assume that Hypotheses \ref{dissiR} hold true with $\zeta_R<0$. If the semigroup $\{P(t)\}_{t\geq 0}$ is supercontractive then the inequalities \eqref{epslog} hold true.
\end{thm}

\begin{proof}
We start by proving the statement for functions $f\in \mathcal{F}C^1_b(H)$ such that $\norm{f}_{L^2(H,\nu)}=1$.
So, let us fix such a function $f$ and $1<p<q<+\infty$. We claim that, there exist a constant $K=K(p,q) \in (0,1)$ and a positive function $\eta_{p,q}:(0,+\infty)\to (0,+\infty)$ blowing up as $t\to 0^+$, independent of $f$, such that
\begin{equation}\label{cliam}
\int_H (P(t)f^2)\ln(P(t)f^2)d\nu \leq K(p,q)\int_H f^2\ln f^2 d\nu+\eta_{p,q}(t),\qquad\;\, t>0.
\end{equation}
To this aim, let us recall that $P(t)$ is supercontractive, then $\|P(t)\|_{\mathcal{L}(L^p(H,\nu);L^q(H,\nu))}\leq c_{p,q}(t)$ for any $t>0$. Moreover, being $\|P(t)\|_{\mathcal{L}(L^1(H,\nu))}\leq 1$ for any $t>0$, by the Riesz--Thorin interpolation theorem (see \cite{Lun_Interpolation}), we deduce that
\[
\norm{P(t)}_{\mathcal{L}(L^{p_h}(H,\nu);L^{q_h}(H,\nu))}\leq [c_{p,q}(t)]^{r_h}
\]
for any $h$ belongs to $[0,1-1/p]$, $r_h=ph(p-1)^{-1}$, $p_h=(1-h)^{-1}$ and $q_h= (1-r_h+r_h/q)^{-1}$.
In particular,
$$\int_H \Big(P(t)\abs{f}^{2(1-h)}\Big)^{q_h}d\nu\le [c_{p,q}(t)]^{r_hq_h}$$
and consequently, for any $h>0$,
$$\frac{1}{h}\left[\int_H \Big(P(t)\abs{f}^{2(1-h)}\Big)^{q_h}d\nu- 1\right]
\leq \frac{1}{h}\left[(c_{p,q}(t))^{r_hq_h}- 1\right].$$
Thus, recalling that $\norm{f}_{L^2(H,\nu)}=1$, by the previous inequality, we get
\begin{align*}
\frac{1}{h}\bigg(\|P(t)\abs{f}^{2(1-h)}\|^{q_h}_{L^{q_h}(H,\nu)}-&\|P(t)\abs{f}^{2}\|_{L^{1}(H,\nu)}\bigg)\leq \frac{1}{h}\left[(c_{p,q}(t))^{r_hq_h}- 1\right].
\end{align*}
Observing that $r_0=0$ and $q_0=1$, and that the functions $h\ra \|P(t)\abs{f}^{2(1-h)}\|^{q_h}_{L^{q_h}(H,\nu)}$ and $h\ra [c_{p,q}(t)]^{r_hq_h}$ are differentiable in $h=0$, by the previous inequality, we obtain
\begin{align*}
\frac{d}{dh}\left(\|P(t)\abs{f}^{2(1-h)}\|^{q_h}_{L^{q_h}(H,\nu)}\right)_{|_{h=0}} \le \frac{d}{dh }(c_{p,q}(t))^{r_hq_h})_{|_{h=0}}
\end{align*}
that, thanks to the invariance of $\nu$, implies
\begin{align*}
\frac{p(q-1)}{q(p-1)}\int_H (P(t)f^2)\ln(P(t)f^2)d\nu & \le \int_H P(t) \left(f^2\ln f^2\right)d\nu+ \frac{p}{p-1}\ln(c_{p,q}(t))\\
& \le \int_H f^2\ln f^2d\nu+ \frac{p}{p-1}\ln(c_{p,q}(t))
\end{align*}
whence the claim with $K(p,q)= \frac{q(p-1)}{p(q-1)}$ and $\eta_{p,q}(t)= K(p,q)\frac{p}{p-1}\ln(c_{p,q}(t))$ which, clearly, blows up as $t \to 0^+$.
To complete the proof, it suffices integrate \eqref{prelog} over $H$ with respect to $\nu$, using its invariance to deduce that
\begin{equation*}
\int_H f^2\ln f^2d\nu\leq  C(t)\int_H\norm{\D_R f}^2_Rd\nu+\int_H P(t)f^2 \ln P(t) f^2d\nu.
\end{equation*}
where $C(t)$ is defined in estimate \eqref{prelog}. Finally, thanks to \eqref{cliam} we conclude that
\begin{align}\label{pre2}
\int_H f^2\ln f^2d\nu&\leq  \frac{p(q-1)}{(q-p)}C(t)\int_H\norm{\D_R f}^2_Rd\nu+\frac{qp}{q-p}\ln(c_{p,q}(t)).
\end{align}
Since $C(t)$ vanishes as $t\to 0^+$ and $c_{p,q}(t)$ blows up as $t\to 0^+$, estimate \eqref{epslog} is proved for functions $f\in \mathcal{F}C^1_b(H)$ with $\norm{f}_{L^2(H,\nu)}=1$.
Condition $\norm{f}_{L^2(H,\nu)}=1$ can be removed by applying \eqref{pre2} to the function $f\|f\|_{L^2(H, \nu)}^{-1}$ to get
\begin{equation}\label{lisa}\int_H f^2\ln\left(\frac{f^2}{\|f\|_{L^2(H, \nu)}^2}\right )d\nu\leq  \frac{p(q-1)}{(q-p)}C(t)\int_H\norm{\D_R f}^2_Rd\nu+\frac{qp}{q-p}\ln(c_{p,q}(t))\|f\|_{L^2(H, \nu)}^2.
\end{equation}
The general case can be obtained by approximation. Indeed, for any $f \in W_R^{1,2}(H, \nu)$ there exists a sequence $\{f_n\}_{n\in\N} \subseteq \mathcal{F}C^1_b(H)$ such that $\|f_n-f\|_{W_R^{1,2}(H, \nu)}$ vanishes as $n \to +\infty$.
The claim will follow writing \eqref{lisa} for $f_n$, letting $n \to \infty$  after observing that
\begin{equation}\label{lisa1}
\lim_{n \to +\infty} \int_H f_n^2\ln f_n^2d\nu= \int_H f^2\ln f^2d\nu.
\end{equation}
The convergence result in \eqref{lisa1} can be proved writing $\ln f_n^2=\ln_{+}f_n^2-|\ln f_n^2|\chi_{\{f_n^2<1\}}$ for any $n \in \N$, where $\ln_+$ denotes the positive part of the natural logarithm. Thus the Fatou lemma and the dominated convergence theorem allow us to deduce \eqref{lisa1}.
\end{proof}

We are now in position to prove a complete characterization of the supercontractivity of $P(t)$ in terms of some families of logarithmic Sobolev inequalities and of some integrability conditions for the functions $x \mapsto e^{\lambda \|x\|_R^2}$ for any $\lambda>0$. We point out that, as already observed, even if a similar characterization is true in finite dimension (see \cite{AngLor}), here we cannot use such results for the finite dimensional approximants, to deduce the same for $P(t)$ simply taking the limit as $n$ approaches infinity.

\begin{thm}\label{sup_thm}
Assume that Hypotheses \ref{dissiR} hold true with $\zeta_R<0$ and $\nu(H_R)=1$. The following statements are equivalent:
\begin{enumerate}[\rm (i)]
\item the semigroup $\{P(t)\}_{t\geq 0}$ is supercontractive;
\item the inequalities \eqref{epslog} hold true;
\item for any $\lambda>0$ it holds
\begin{equation}\label{exp2}
\int_{H_R}e^{\lambda\norm{x}^2_R}\nu(dx)<+\infty.
\end{equation}
\end{enumerate}
\end{thm}
\begin{proof}
By Theorem \ref{super-eps} we known (i) implies (ii). Observe that (ii) implies (iii) follows immediately by Proposition \ref{code1} taking $g(x)=\|x\|_R$.

To prove that (iii) implies (i) it suffices to show that $\norm{P(t)f}^q_{L^q(H,\nu)}\leq M_{p,q}(t)\norm{f}_{L^p(H,\nu)}$ for any $f \in C_b(H)$, any $t>0$ and some  function $M_{p,q}:(0,+\infty)\to(0,+\infty)$. The general case when $f$ belongs to $L^p(H, \nu)$ can be obtained by approximation and the dominated convergence theorem.
Let us consider $f\in C_b(H)$, $1<p<+\infty$ and $t>0$. By using the Harnack inequality in  \eqref{dae}, the invariance of $\nu$ and the fact that $\nu(H_R)=1$ we deduce
\begin{align*}
\int_H\abs{f(y)}^p \nu(dy)&=\int_H P(t)\abs{f}^p(y) \nu(dy)=\int_{H_R}P(t)\abs{f}^p(y) \nu(dy)\\
&\geq  \abs{P(t)f(x)}^p\int_{H_R}{\rm exp}\pa{-p(p-1)^{-1}K(t)\|x-y\|_R^2}\nu(dy)\\
&\geq  \abs{P(t)f(x)}^p\nu(B_R(0,r)){\rm exp}\pa{-\frac{p(e^{2\zeta_R t}-1)}{2\zeta_R(p-1) t^2}(r^2+\|x\|_R^2)},
\end{align*}
where $B_R(0,r)$ is the ball in $H_R$ with center $0$ and radius $r>0$. Hence for any $x\in H_R$ we have
\begin{align}\label{S1}
\abs{P(t)f(x)}\leq 2{\rm exp}\pa{\frac{e^{2\zeta_R t}-1}{2\zeta_R(p-1) t^2}(\overline{r}^2+\|x\|_R^2)}\norm{f}_{L^p(H,\nu)}
\end{align}
choosing $\overline{r}=\overline{r}_p>0$ such that $\nu(B_R(0,\overline{r}))>2^{-p}$. Let $1<p<q<+\infty$. By \eqref{S1} we have
\begin{align*}
\norm{P(t)f}^q_{L^q(H,\nu)}&=\int_H\abs{P(t)f(x)}^q\nu(dx)=\int_{H_R}\abs{P(t)f(x)}^q\nu(dx)\\
&\leq 2^q\left(\int_{H_R}{\rm exp}\pa{\frac{q(e^{2\zeta_R t}-1)}{2\zeta_R(p-1) t^2}(\overline{r}^2+\|x\|_R^2)}\nu(dx) \right)\norm{f}_{L^p(H,\nu)}^q.
\end{align*}
Thanks to \eqref{exp2}
\[
(M_{p,q}(t))^q:=2^q\int_{H_R}{\rm exp}\pa{\frac{q(e^{2\zeta_R t}-1)}{2\zeta_R(p-1) t^2}(\overline{r}^2+\|x\|_R^2)}\nu(dx)<+\infty,
\]
and, consequently, for any $f\in C_b(H)$, it holds $\norm{P(t)f}_{L^q(H,\nu)}\leq M_{p,q}(t)\norm{f}_{L^p(H,\nu)}$, whence the claim.
\end{proof}

\begin{rmk}\label{BohBoh}{\rm
\begin{enumerate}[\rm (i)]
\item
As already pointed out in the introduction, even if it should be quite natural to exploit \cite[Theorem 3.1]{AngLor} and the approximation in \eqref{fin_approx} to prove  directly the equivalence (i) $\Leftrightarrow$ (ii) in Theorem \ref{sup_thm}, this approach does not work. Indeed, for every $n\in\N$ and $\delta>0$ the semigroup $\{P_{\delta,n}(t)\}_{t\geq 0}$ admits a unique invariant measure $\nu_{\delta,n}$, but in our general setting, the relation between $\nu_{\delta,n}$ and $\nu$ is not explicit and then it is not possible to relate the supercontractivity property of $\{P(t)\}_{t\geq 0}$ to that of $\{P_{\delta,n}(t)\}_{t\geq 0}$ for any $n\in\N$ and $\delta>0$ and viceversa.
\item
Note that in the case $H=\R^n$, in \cite[Lemma 3.8]{AngLor} the authors provide sufficient conditions on the drift of \eqref{eqF02} to get property (iii) of Theorem \ref{sup_thm} and consequently, to deduce the supercontractivity of $\{P(t)\}_{t\geq 0}$. However, these conditions require the existence of a Lyapunov function for the second order Kolmogorov operator associated to \eqref{eqF02}. In the infinite dimensional setting, the theory of Lyapunov functions is not well developed also because these arguments are based on estimates that depend explicitly on the dimension and that cannot be extended to the infinite dimension.

\item
In many cases, the assumption $\nu(H_R)=1$ is checked assuming that the mild solution of the stochastic partial differential equation \eqref{eqF02} takes values in $H_R$.
In these situations, Theorem \ref{sup_thm} can be applied with $R=\Id$ and replacing $H$ with $H_R$. 
%
%
However, we did not choose this approach to highlight the fact that the hypothesis $\nu(H_R)=1$ is just needed to prove in Theorem \ref{sup_thm} that (iii) implies (i). We believe that it is possible to avoid the condition $\nu(H_R)=1$ if one were able to prove a Harnack inequality of the type \eqref{dae} with $\norm{\cdot}_R$ replaced by $\norm{\cdot}_H$. This type of Harnack inequality has already been proven in the linear case in \cite{Ouy-Roc1}, via a technique that is not applicable in our case. The main problem considering stochastic partial differential equation of the type \eqref{cerrai} is proving a pointwise version of the estimate in \cite[Proposition 6.4.1]{CER1}.

\end{enumerate}}
\end{rmk}

\section{Applications to ultraboundness}\label{IDP0}
Here, we collect some consequences of the results in the previous sections. We start finding conditions that ensure summability properties stronger than the supercontractivity. In particular, the following result allows us to characterize the ultraboundedness of the semigroup $P(t)$, i.e. its boundedness as an operator from $L^p(H, \nu)$ into $L^\infty(H, \nu)$ for any $p>1$ and any $t>0$.

\begin{prop}\label{ultra}
Assume that Hypotheses \ref{dissiR} hold true with $\zeta_R < 0$, $\nu(H_R) = 1$ and consider the following assertions:
\begin{enumerate}[\rm(i)]
\item the semigroup $P(t)$ is ultrabounded, i.e., for any $p > 1$ and $t > 0$:
\[\|P(t)\|_{\mathcal{L}(L^p(H,\nu);L^\infty(H,\nu))}<+\infty;\]

\item for any $\delta>0$ and $\lambda > 0$, there exists $c_{\delta, \lambda} > 0$ such that $\|P(t) \varphi_\lambda\|_{L^\infty(H,\nu)} \leq c_{\delta, \lambda}$ for any $t \geq \delta$ where $\varphi_\lambda(x) := e^{\lambda \|x\|^2_R}$.
\end{enumerate}
Then $(i) \Rightarrow (ii)$. Further, if $P(t)$ is supercontractive then $(ii) \Rightarrow(i)$.
\end{prop}

\begin{proof}
To begin with, we observe that if $P(t)$ is ultrabounded, then it is supercontractive too. According to Theorem \ref{sup_thm} the functions $\varphi_{\lambda}$ belong to $L^p(H, \nu)$ for any $\lambda>0$ and $p \ge1$. As a consequence, $P(t)\varphi_\lambda$ is well defined $\nu$-a.e. in $H$. To prove that (i) implies (ii), it is enough to observe that
\[\|P(t)\varphi_\lambda\|_{L^\infty(H,\nu)} \le \|P(t)\|_{\mathcal{L}(L^2(H, \nu),L^\infty(H, \nu))}\|\varphi_\lambda\|_{L^2(H, \nu)}\]
Since $t\mapsto\|P(t)\|_{\mathcal{L}(L^2(H, \nu),L^\infty(H, \nu))}$ blows up at $t=0$ and is decreasing in $(0,+\infty)$ we get that $\|P(t)\varphi_\lambda\|_\infty\le c_{\delta, \lambda}$ for any $t \ge \delta$ and some positive constant $c_{\delta, \lambda}$.

Conversely, if $P(t)$ is supercontractive, in order to prove that (ii) implies (i) we take $f \in L^p(H, \nu)$ and for any $\delta>0$, let $M_\delta$ such that $2e^{\zeta_R t}\ge 2+ \zeta_R M_\delta t^2 $ for any $t \ge \delta$.
Utilizing both the semigroup law and inequality \eqref{S1}, we can then deduce that:
\begin{align*}
|P(t)f(x)|=& |P(t/2)P(t/2)f(x)|\le  P(t/2)|P(t/2)f(x)|\\
\le&  P(t/2)\left(2{\rm exp}\pa{\frac{2(e^{\zeta_R t}-1)}{\zeta_R(p-1)t^2}(\overline{r}^2+\|x\|_R^2)}\norm{f}_{L^p(H,\nu)}\right)\\
\le & 2{\rm exp}\pa{\frac{2(e^{\zeta_R t}-1)}{\zeta_R(p-1)t^2}\overline{r}^2}\|f\|_{L^p(H,\nu)}\pa{P(t/2)\pa{{\rm exp}\pa{\frac{2(e^{\zeta_R t}-1)}{\zeta_R(p-1)t^2}(\norm{\cdot}_R^2)}}}(x)\\
\le & 2c_{\delta, (p-1)^{-1}M_\delta}{\rm exp}\pa{\frac{2(e^{\zeta_R t}-1)}{\zeta_R(p-1)t^2}\overline{r}^2}\|f\|_{L^p(H,\nu)},
\end{align*}
for any $t \ge \delta$ where $\overline{r}$ is chosen as in the proof of Theorem \ref{sup_thm}. The arbitrariness of $\delta$ allows us to conclude.
\end{proof}

\subsection{Sufficient conditions for ultraboundness}

We assume the following additional hypotheses on the perturbation $F$.

\begin{hyp}\label{super-dissiR}
Under Hypotheses \ref{dissiR}, assume that $A_R$ is a negative operator and that it holds $F_{|_E}(x)(H_R\cap E)\subseteq H_R$. Suppose further that there exist $a>0$ and a strictly increasing function $\phi:[0,+\infty)\ra [0,+\infty)$ belonging to $C([0,+\infty))$ such that $1/\phi\in L^1([b,+\infty))$, for every $b>0$, $1/\phi\notin L^1([0,+\infty))$, and
\begin{align}
\lim_{s\rightarrow+\infty}s^{-1}\phi(s) &=+\infty,\\
\scal{F_E(x)-F_E(y)}{x-y}_R &\leq a-\phi(\norm{x-y}^2_R), \qquad x,y\in H_R\cap E.\label{YIIK}
\end{align}
\end{hyp}

\begin{prop}
Assume Hypotheses \ref{super-dissiR} hold true. For every $x,y\in H_R$ and $t>0$, it holds
\begin{align}\label{super-stima}
\norm{X(t,x)-X(t,y)}^2_R\leq 2\phi^{-1}(2a)+\psi^{-1}(t/4),\qquad \qc
\end{align}
where $\psi(s):=\int_{s}^{+\infty} 1/\phi(r)dr$.
\end{prop}

\begin{proof}
It is not restrictive to assume that both $\{X(t,x)\}_{t\geq 0}$ and $\{X(t,y)\}_{t\geq 0}$ are strict solutions of \eqref{SPDE} (with initial datum $x$ and $y$, respectively). Otherwise we proceed as in \cite[Proposition 3.6]{BI1} or \cite[Proposition 6.2.2]{CER1} approximating $\{X(t,x)\}_{t\geq 0}$ and $\{X(t,y)\}_{t\geq 0}$ by means of sequences of more regular processes using the Yosida approximation of $A_R$. Estimates similar to \eqref{stimaE} and \eqref{persobolev} yield that $X(t,x)-X(t,y)\in E\cap H_R$ for any $t>0$, $x,y\in H_R\cap E$.
Now, let us fix $x, y\in E\cap H_R$, multiplying the equation solved by $X(t,x)-X(t,y)$ by the difference $X(t,x)-X(t,y)$ and using Hypotheses \ref{super-dissiR}, we obtain
\begin{align*}
\frac{1}{2}\frac{d}{dt}\|X(t,x)-X(t,y)\|^2_R\leq a-\phi\left(\norm{X(t,x)-X(t,y)}^2_R\right).
\end{align*}
Aguing as in \cite[pp. 1010-1011]{DaPROWA09} and using the density of $E\cap H_R$ in $H_R$ conclude the proof.
\end{proof}

\begin{prop}\label{C-ultra}
Assume that Hypotheses \ref{super-dissiR} hold true, $\nu(H_R)=1$, and that $P(t)$ is supercontractive. Then $P(t)$ is ultrabounded.
\end{prop}

\begin{proof}
By Proposition \ref{ultra} it is sufficient to prove that for any $\delta>0$ and $\lambda > 0$, there exists $c_{\delta, \lambda} > 0$ such that $\|P(t) \varphi_\lambda\|_{L^\infty(H,\nu)} \leq c_{\delta, \lambda}$ for any $t \geq \delta$ where $\varphi_\lambda(x) := e^{\lambda \|x\|^2_R}$.
For every $\beta>0$, we consider the function
\[
\varphi_{\lambda,\beta}(x)=\dfrac{e^{\lambda \|x\|^2_R}}{1+\beta e^{\lambda \|x\|^2_R}}.
\]
By the invariance of $\nu$ and the Jensen inequality we have
\[
\norm{P(s)\varphi_{\lambda,\beta}}_{L^p(H,\nu)}^p=\int_H(P(s)\varphi_{\lambda,\beta})^pd\nu\leq \int_HP(s)(\varphi_{\lambda,\beta})^pd\nu=\int_H(\varphi_{\lambda,\beta})^pd\nu,
\]
by the Monotone convergence theorem and Theorem \ref{sup_thm} we obtain
\[
\norm{P(s)\varphi_{\lambda}}_{L^p(H,\nu)}^p\leq \int_H\varphi_{p\lambda}d\nu<+\infty
\]
so $P(s)\varphi_\lambda\in L^p(H,\nu)$ for every $s\geq 0$, $\lambda>0$ and $p>1$.

By Theorem \ref{dae}, for every $\beta,s,\lambda\geq 0$ and $x,y\in H$ such that $x-y\in H_R$ we get
\begin{align*}
(P(s)\varphi_{\lambda,\beta}(x))^2\leq P(s)\varphi_{\lambda,\beta}^2(y)e^{(e^{2\zeta_Rt}-1)(\zeta_Rt^2)^{-1}\norm{x-y}^2_R}
\end{align*}
letting $\beta\rightarrow 0$ by the monotone convergence theorem we obtain
\begin{align}\label{ss1}
(P(s)\varphi_{\lambda}(x))^2\leq P(s)\varphi_{2\lambda}(y)e^{(e^{2\zeta_Rt}-1)(\zeta_Rt^2)^{-1}\norm{x-y}^2_R}
\end{align}
By \eqref{super-stima}, \eqref{ss1} and the Jensen inequality, for every $\beta,s,\lambda\geq 0$ and $x,y\in H_R$ we obtain
\begin{align*}
(P(t)\varphi_\lambda(x))^2\leq P(t/2)\left[P(t/2)\varphi_\lambda\right]^2(x)\leq P(t)\varphi_{2\lambda}(y)e^{4(e^{\zeta_Rt}-1)(\zeta_Rt^2)^{-1}(2\phi^{-1}(2a)+\psi^{-1}(t/8))}.
\end{align*}
By Hypotheses \ref{super-dissiR} we have that $\psi^{-1}(t)\ra 0$ as $t$ approaches $+\infty$. Hence
\[
\lim_{t\rightarrow +\infty}e^{4(e^{\zeta_Rt}-1)(\zeta_Rt^2)^{-1}(2\phi^{-1}(2a)+\psi^{-1}(t/8))}=1
\]
so there exists $\delta, c_\delta>0$ such that for every $t>\delta$
\begin{align*}
(P(t)\varphi_\lambda(x))^2\leq c_\delta P(t)\varphi_{2\lambda}(y),
\end{align*}
finally integrating both sides with respect to $\nu(dy)$, by the invariance of $\nu$ and Theorem \ref{sup_thm} we get
\begin{align*}
(P(t)\varphi_\lambda(x))^2\leq c_\delta\norm{\varphi_{2\lambda}}_{L^2(H,\nu)}^2, \quad t>\delta,
\end{align*}
so by the equivalence Theorem \ref{ultra} the semigroup $P(t)$ is ultrabounded.
\end{proof}

\section{Examples}\label{exp0}
Here we will provide examples of problems like \eqref{SPDE} where our results can be applied.

\subsection{Infinite dimensional polynomial}
Let $E=H=L^2([0,1]^d,\lambda)$ with $d\in\N$ and $\lambda$ is the Lebesgue measure on $[0,1]^d$. Let $A$ be the realization of the Laplacian operator in $H$ with Dirichlet boundary condition. Let $F:H\rightarrow H$ be a function defined by
\begin{align*}
[F(f)](\xi):=\zeta_Ff(\xi)+\int_0^1\int_0^1\int_0^1 K(\xi_1,\xi_2,\xi_3,\xi) f(\xi_1)f(\xi_2)f(\xi_3)d\xi_1 d\xi_2 d\xi_3,\qquad \xi\in [0,1]^d,
\end{align*}
where $\zeta_F\in\R$ and $K\in L^2([0,1]^{3+d},\lambda)$. For a fixed $\beta>1/2$, we consider the following stochastic partial differential equation
\begin{gather*}
\eqsys{
dX(t)=\left[-(-A)^{\beta}X(t)+F(X(t))\right]dt+(-A)^{-1/2}dW(t), & t>0;\\
X(0)=x\in L^2([0,1]^d,\lambda).
}
\end{gather*}
In this case $R=(-A)^{-1/2}$ and $H_R=W_0^{1,2}([0,1]^d,\lambda)$, i.e. the set of functions belonging to $W^{1,2}([0,1]^d,\lambda)$ and having null trace on the boundary of $[0,1]^d$.
Let $\{\lambda_k\}_{k\in\N}$ be the sequence of eigenvalues of $A$ in $H$. Recall that, for every $d\in\N$ we have
\begin{equation}\label{autoval-laplaciano}
\lambda_k\approx -k^{2/d},\qquad k\in\N.
\end{equation}
Note that, thanks to \eqref{autoval-laplaciano}, Hypothesis \ref{main}(iv) holds true for every $\beta>d/2-1$. Moreover assuming that $(\partial K/\partial \xi)\in L^2([0,1]^{3+d},\lambda)$ and that $K$ and $(\partial K/\partial \xi)$ have symmetric versions (see \cite[Formula (5.10)]{ABF21}) then all the assumptions in Theorems \ref{logsob_pro} and \ref{Hyper} are satisfied. We refer to  \cite[Examples 5.3 and 5.4]{ABF21} for detailed calculations.

\subsection{Radial perturbation}
Let $f:[0,+\infty)\rightarrow [0,+\infty)$ be an increasing and differentiable function.
We consider the function $U:H\rightarrow \R$ defined by
\begin{equation}\label{U}
U(x)=-f(\norm{x}_H^2),\qquad x\in H.
\end{equation}
and the transition semigroup $P(t)$ associated to
\begin{align}
\eqsys{dX(t)=[AX(t)+\nabla U(X(t))]dt+dW(t) \qquad\;\, & t>0,\\
X(0)=x \in H,}
\end{align}
where $A$ is the realization of the second order derivative with Dirichlet boundary condition in $H=L^2([0,1])$.
It is well known that $P(t)$ has a unique invariant measure given by
\[
\nu(dx)=e^U\mu(dx),\qquad \mu\sim\mathcal{N}(0,(-A)^{-1}).
\]
We now want to find sufficient conditions that guarantee the supercontractivity and ultraboundness of $P(t)$.

\subsubsection*{Supercontractivity}
Here we prove that if
\begin{equation}\label{limite}
\lim_{s\rightarrow+\infty}f(s)s^{-1}=+\infty,
\end{equation}
then $P(t)$ is supercontractive, i.e. $\int_He^{\lambda\|x\|^2}d\nu<+\infty$ for any $\lambda>0$.
To this aim, let $\alpha_0>0$ be the optimal exponent given by Fernique Theorem such that for every $\delta<\alpha_0$ it holds
\begin{equation}\label{Fernique}
\int_He^{\delta\norm{x}_H^2}\mu(dx)<+\infty.
\end{equation}
For every $\lambda>0$ we have
\begin{align}\label{VoidStranger}
\int_{H}e^{\lambda\norm{x}^2}\nu(dx) &=\int_{H}e^{\lambda\norm{x}^2}e^{U}\mu(dx)\notag\\
&=\int_{H}e^{\lambda\norm{x}^2_H-f(\norm{x}_H^2)}\mu(dx)=\int_{H}e^{\alpha_0\norm{x}^2_H\left(\frac{\lambda}{\alpha_0}-\frac{f(\norm{x}_H^2)}{\alpha_0\norm{x}^2_H}\right)}\mu(dx)
\end{align}
By \eqref{limite} there exists $r_0>0$ such that for every $x\in H$ with $\|x\|^2>r_0$ it holds
\begin{align}\label{JWB}
f(\norm{x}_H^2)>(\lambda-\alpha_0)\|x\|^2.
\end{align}
Let $B_{\sqrt{r_0}}$ the ball of $H$ with radius $\sqrt{r_0}$. By \eqref{Fernique}, \eqref{VoidStranger} and \eqref{JWB} for every $\lambda>0$ there exists $\delta<\alpha_0$ such that
\begin{align}\label{Fernique+}
\int_{H}e^{\lambda\norm{x}^2}e^{U}\mu(dx)\leq e^{\lambda r_0}+\int_{H\backslash B_{\sqrt{r_0}}}e^{\delta\norm{x}^2}\mu(dx)<+\infty
\end{align}
Hence by \eqref{Fernique+} and Theorem \ref{sup_thm}, $P(t)$ is supercontractive.

\subsubsection*{Ultraboundness}
Here we assume that $\lim_{s\ra+\infty}f(s)=+\infty$ and that
\begin{align}
-2\langle f'(\norm{x}_H^2)x-f'(\norm{y}_H^2)y,x-y\rangle &\leq a-\phi(\norm{x-y}^2), \qquad x,y\in H.\label{FYIIK}
\end{align}
where  $a>0$ and $\phi:[0,+\infty)\ra [0,+\infty)$ is a strictly increasing function belonging to $C([0,+\infty))$ such that $1/\phi\in L^1([b,+\infty))$, for every $b>0$, $1/\phi\notin L^1([0,+\infty))$. It easy to see that the function $U$ defined in \ref{U} verifies Hypotheses \ref{super-dissiR} with $R=\Id_H$.

We claim that $\lim_{s\ra+\infty}s^{-1}f(s)=+\infty$. 
Taking $y=0$ in \eqref{FYIIK} we obtain
\begin{align}\label{stecca}
-2 f'(\norm{x}_H^2)\norm{x}_H^2\leq a-\phi(\norm{x}_H^2).
\end{align}
Setting $r=\norm{x}^2$ in \eqref{stecca} and dividing both members by $-2r$ we get
\begin{align}\label{stecca2}
f'(r)\geq -\frac{a}{2r}+\frac{\phi(r)}{2r}.
\end{align}
By the l'H\^opital rule, the fact that $\lim_{s\ra+\infty}f(s)=+\infty$ and \eqref{stecca2} we get
\begin{align*}
\lim_{s\ra+\infty}\frac{f(s)}{s}=+\infty.
\end{align*}
So the claim in proved. Hence, by the first part of the example, $P(t)$ is supercontractive and by Proposition \ref{C-ultra} the semigroup $P(t)$ is ultrabounded, too.

\subsection{A reaction--diffusion system}\label{RDE}
We consider the following stochastic reaction-diffusion equation:
\begin{align}\label{eqF02}
\begin{cases}
dX(t)= \left[\frac{\partial^2}{\partial\xi^2}X(t) - b(X(t))\right]dt + dW(t), & t > 0;\\
X(t)(0)= X(t)(1), & t > 0;\\
X(0) = x\in L^2([0,1],\lambda).
\end{cases}
\end{align}
where $\lambda$ is the Lebesgue measure on $[0,1]$ and $b:\R\ra\R$ is a smooth enough function satisfying the following conditions:

\begin{hyp}\label{RDS1}
The function $b:\R\ra\R$ belongs to $C^3(\R)$ and
\begin{enumerate}[\rm (i)]
\item there exists an integer $m\geq 0$ such that for every $j=0,1,2,3$ it holds
\begin{align*}
\sup_{z\in \R}\frac{\abs{\J^j_zb(z)}}{1+|z|^{\max\{0,2m+1-j\}}}<+\infty,
\end{align*}
where $\J^0_zb$ denotes simply the function $b$;

\item there exists $a>0$ such that for every $z,h\in \R$ it holds $(b(z+h)-b(z))h\leq -a h^{2}$.

\end{enumerate}
\end{hyp}
\noindent A simple example of function $b$ that verifies Hypotheses \ref{RDS1} is
\begin{equation}\label{b-piccolo}
b(z):=C_{2m+1}z^{2m+1}+\sum^{2m}_{k=0}C_kz^k, \qquad z\in\R,
\end{equation}
where $C_0,\ldots,C_{2m}\in\R$ and $C_{2m+1}>0$. The stochastic equation \eqref{eqF02} can be rewritten in an abstract way as \eqref{SPDE} where $A$ is the realization in $L^2([0,1],\lambda)$ of the second order derivative with Dirichlet boundary conditions, $R=\Id_{L^2([0,1],\lambda)}$ and $F:C([0,1])\ra C([0,1])$ is the Nemytskii operator defined as
\[
F(x)(\xi):=-b(x(\xi)),\qquad \xi\in [0,1],\; x\in C([0,1]).
\]
Note that Hypotheses \ref{main} (with $E=C([0,1])$) are satisfied for the equation \eqref{eqF02}, see  \cite[Chapter 6]{CER1}) and since $R=\Id_{L^2([0,1],\lambda)}$ also Hypotheses \ref{hyp2} and \ref{dissiR} hold true. Hence all the results in this paper can be applied to the transition semigroup $P(t)$ associated to \eqref{eqF02}. Finally we show that under an additional assumption on $b$ the invariant measure $\nu$ of $P(t)$ verifies the condition in Theorem \ref{sup_thm}(iii).

Let $B:\R\ra\R$ be a primitive of $b$ and consider $U:L^2([0,1],\lambda)\ra \R$ given by
\[
U(x)=\begin{cases}
\int^1_0 B(x(\xi))d\xi, & x\in C([0,1]);\\
0, & x\not\in C([0,1]).
\end{cases}
\]
Let $\mu\sim N(0,(-A)^{-1})$. By \cite[Proposition 5.2]{DA-LU2}, the function $U$ belongs to $W^{1,p}(L^2([0,1],\lambda),\mu)$, for any $p\geq 1$, and
\[
\nabla U(x)(\xi)=b(x(\xi))=-F(x)(\xi),\qquad x\in C([0,1]),\ \xi\in[0,1].
\]
Since $U$ is a potential of $-F$, then the invariant measure $\nu$ of $P(t)$ is the Gibbs measure given by
\[
d\nu=e^{-U}d\mu.
\]
Hence if we assume that there exist $c,\eps>0$ such that for any $\xi\in\R$ it holds that
\begin{equation}\label{condB}
B(\xi)\geq c|\xi|^{2+\eps},
\end{equation}
then for any $\lambda>0$ we have
\begin{align*}
\int_{L^2([0,1],\lambda)}e^{\lambda||x||^2_{L^2([0,1],\lambda)}}\nu(dx)&=\int_{L^2([0,1],\lambda)} \exp\left(||x||^2_{L^2([0,1],\lambda)}-\int^1_0B(x(\xi))d\xi\right)\mu(dx)\\
&\leq \int_{L^2([0,1],\lambda)} \exp\left(||x||^2_{L^2([0,1],\lambda)}-c\int^1_0|x(\xi)|^{2+\eps}d\xi\right)\mu(dx)\\
&\leq \int_{L^2([0,1],\lambda)} \exp\left(||x||^2_{L^2([0,1],\lambda)}-c\left(\int^1_0|x(\xi)|^{2}d\xi\right)^{(2+\eps)/2}\right)\mu(dx)\\
&=\int_{L^2([0,1],\lambda)} \exp\left(||x||^2_{L^2([0,1],\lambda)}-c||x||^{2+\eps}_{L^2([0,1],\lambda)}\right)\mu(dx)<+\infty
\end{align*}
and so $\nu$ verifies the conditions in Theorem \ref{sup_thm}(iii) and $P(t)$ is supercontractive. Observe that condition \eqref{condB} is satisfied if $m\geq 1$ (see Hypothesis \ref{RDS1}(i)). Moreover we stress that if for example
\[
b(z)=Cz^{2m+1}
\]
for some $m\in\N$ and $C>0$, then the function $F$ verify Hypotheses \ref{super-dissiR} with $\phi(s)\approx s^{2m+2}$. Hence the assumptions of Proposition \ref{C-ultra} hold true and the semigroup $P(t)$ is ultrabounded.

\vspace{0.6cm}

\section*{Statements and Declarations}

\noindent {\bf Competing Interests.} The authors do not have financial or non-financial interests that are directly or indirectly related to the work submitted for publication.

\noindent {\bf Data Availability Statements.} Data sharing not applicable to this article as no datasets were generated or analysed during the current study.

\appendix

\section{A Lasry--Lions type approximations}\label{LASRY}

In this appendix, we recall some properties of the Lasry-Lions type approximants along $H_R$ (see \cite{BF23,BFFZ23}). The construction of such approximants is a modification of the original idea presented in \cite{LL86}.
For $\eps>0$ and $f\in \mathrm{BUC}(H)$, we consider the functions defined as follows:
\begin{align}\label{LLCD}
f_\eps(x):=\sup_{h\in H_R}\set{\inf_{k\in H_R}\set{f(x+k-h)+\frac{1}{2\eps}\|k\|_{R}^2}-\frac{1}{\eps}\|h\|_{R}^2},\qquad x\in H.
\end{align}
In the following proposition, we summarize some basic properties of the approximants defined in \eqref{LLCD} that will be useful in the paper. Some of the results we will present are already contained in \cite{BF23} and \cite{BFFZ23}. However, we provide a proof of them to make the paper self-contained.

\begin{prop}\label{Prop_LASRY}
Assume Hypothesis \ref{main}(ii) and let $f\in {\rm Lip}_{H_R}(H)$. Let $\{f_\eps\}_{\eps\geq 0}$ be the family of functions introduced in \eqref{LLCD}. For every $\eps>0$ and $x\in H$ it holds
\begin{align}
\|f_\eps\|_\infty&\leq\|f\|_\infty;\label{LL_limitatezza}\\
0\leq f(x)-f_\eps(x)&\leq 4\eps({\rm Lip}_R(f))^2;\label{LL_approximation}\\
\|\nabla_Rf_\eps\|_\infty &\leq 4\sqrt{2}{\rm Lip}_R(f).\label{LL_derivative_estimate}
\end{align}
\end{prop}

\begin{proof}
We start by proving \eqref{LL_limitatezza}.
\begin{align}
f_\eps(x)&=\sup_{h\in H_R}\set{\inf_{k\in H_R}\set{f(x+h-k)+\frac{1}{2\eps}\|k\|_{R}^2}-\frac{1}{\eps}\|h\|_{R}^2}\notag\\
&\leq \sup_{h\in H_R}\set{f(x)+\frac{1}{2\eps}\|h\|_{R}^2-\frac{1}{\eps}\|h\|_{R}^2}\leq f(x)\leq \|f\|_\infty.\label{Gerald1}
\end{align}
In a similar way
\begin{align}
f_\eps(x)&=\sup_{h\in H_R}\set{\inf_{k\in H_R}\set{f(x+h-k)+\frac{1}{2\eps}\|k\|_{R}^2}-\frac{1}{\eps}\|h\|_{R}^2}\notag\\
&\geq \inf_{k\in H_R}\set{f(x-k)+\frac{1}{2\eps}\|k\|_{R}^2}\geq -\|f\|_\infty.\label{Gerald2}
\end{align}
By \eqref{Gerald1} and \eqref{Gerald2} we get \eqref{LL_limitatezza}.

Now we prove \eqref{LL_approximation}. By \eqref{LLCD}, for every $\eta>0$ there exists $k_\eta\in H_R$ such that
\begin{align}
0\leq f(x)-f_\eps(x)&\leq f(x)-\inf_{k\in H_R}\set{f(x-k)+\frac{1}{2\eps}\|k\|_{R}^2}\notag\\
&\leq f(x)-f(x-k_\eta)-\frac{1}{2\eps}\|k_\eta\|_{R}^2+\eta\notag\\
&\leq {\rm Lip}_R(f)\|k_\eta\|_{R}-\frac{1}{2\eps}\|k_\eta\|_{R}^2+\eta.\label{Regan}
\end{align}
By \eqref{Regan} we get the estimate
\begin{align*}
\|k_\eta\|_{R}^2\leq 2\eps {\rm Lip}_R(f)\|k_\eta\|_{R}+2\eps \eta,
\end{align*}
and the Young inequality yields
\begin{align}\label{Truman}
\|k_\eta\|_{R}^2\leq 4\eps^2 ({\rm Lip}_R(f))^2+4\eps\eta.
\end{align}
Combining \eqref{Regan} and \eqref{Truman} we obtain
\begin{align*}
0\leq f(x)-f_\eps(x)&\leq {\rm Lip}_R(f)(4\eps^2 ({\rm Lip}_R(f))^2+4\eps\eta)^{1/2}+2\eps({\rm Lip}_R(f))^2+3\eta.
\end{align*}
Since the above estimate holds for every $\eta>0$, by choosing $\eta$ arbitrarily small, we get \eqref{LL_approximation}.

Let us now prove \eqref{LL_derivative_estimate}. The differentiability along $H_R$ of $f_\eps$ is classical and can be found in \cite{BF23}. By \eqref{LLCD} for every $\sigma>0$ there exists $h_\sigma\in H_R$ such that
\begin{align*}
f_\eps(x)\leq \inf_{k\in H_R}\set{f(x+h_\sigma-k)+\frac{1}{2\eps}\|k\|_{R}^2}-\frac{1}{\eps}\|h_\sigma\|_{R}^2+\sigma.
\end{align*}
A straightforward calculation gives
\begin{align*}
\frac{1}{\eps}\|h_\sigma\|_{R}^2\leq f(x)-f_\eps(x)+\sigma+\frac{1}{2\eps}\|h_\sigma\|_{R}^2.
\end{align*}
Thus from \eqref{LL_approximation} we obtain
\begin{align}\label{Discipline}
\|h_\sigma\|_{R}^2\leq 8\eps^2({\rm Lip}_R(f))^2+2\eps\sigma.
\end{align}
By \eqref{Discipline} we get
\begin{align*}
f_\eps(x+h)-f_\eps(x)&\leq \inf_{k\in H_R}\set{f(x+h+h_\sigma-k)+\frac{1}{2\eps}\|k\|_{R}^2}-\frac{1}{\eps}\|h_\sigma\|_{R}^2+\sigma\\
& -\inf_{k\in H_R}\set{f(x+h+h_\sigma-k)+\frac{1}{2\eps}\|k\|_{R}^2}+\frac{1}{\eps}\|h+h_\sigma\|_{R}^2\\
&=\frac{1}{\eps}\|h+h_\sigma\|_{R}^2-\frac{1}{\eps}\|h_\sigma\|_{R}^2+\sigma=\frac{1}{\eps}\|h\|_{R}^2+\frac{2}{\eps}\langle h,h_\sigma\rangle_{R}+\sigma\\
&\leq \frac{1}{\eps}\|h\|_{R}^2+\frac{2}{\eps}\|h\|_{R}(8\eps^2({\rm Lip}_R(f))^2+2\eps\sigma)^{1/2}+\sigma.
\end{align*}
Since the above inequality holds for every $\sigma>0$, taking the infimum we get
\begin{align}\label{senza_mod}
f_\eps(x+h)-f_\eps(x)\leq \frac{1}{\eps}\|h\|_{R}^2+4\sqrt{2}\|h\|_{R}{\rm Lip}_R(f).
\end{align}
Arguing similarly, we deduce the same inequality for $f_\eps(x)-f_\eps(x+h)$. In this way estimate \eqref{senza_mod} holds true for $|f_\eps(x+h)-f_\eps(x)|$. Thus, dividing by $\|h\|_R$ and letting $\|h\|_R\to 0$ we complete the proof.
\end{proof}

\end{document}